\documentclass[10 pt,a4paper,oneside]{article}
\usepackage[utf8]{inputenc}
\usepackage[OT1]{fontenc}
\usepackage[english]{babel}
\usepackage{dsfont}
\usepackage{amsfonts}
\usepackage{amsmath}
\usepackage{wasysym}
\usepackage{amsthm}
\usepackage{amssymb}
\usepackage{amsthm}
\usepackage{enumitem}
\usepackage{float}
\usepackage{textcomp}
\usepackage{xcolor}
\usepackage{graphicx}
\usepackage{wrapfig}
\usepackage{tikz-cd}
\usepackage{tikz}
\usepackage{mathdots}
\usepackage{fancybox}
\usepackage{fancyhdr}
\usepackage{mathrsfs}

\usepackage[all]{xy}
\usepackage{centernot}
\usepackage{listings}
\usepackage{faktor}
\usepackage{bbm}
\usepackage{bm}
\usepackage{hyperref}
\usepackage{newlfont}
\usepackage{geometry}
\usepackage{ccaption}
\usetikzlibrary{matrix,arrows,decorations.pathmorphing, calc}

\usepackage[style=alphabetic,sorting=nyt,maxbibnames=9]{biblatex}
\usepackage{guit}
\usepackage{calligra,amsmath,amssymb}
\usepackage{calc}
\usepackage{lmodern}

\usepackage{ytableau}
\usepackage{subcaption}
\usepackage{tocloft}

\usepackage{hhline}

\newcounter{daggerfootnote}

\newtheorem{prop}{Proposition}[subsection]
\newtheorem{lemma}[prop]{Lemma}
\newtheorem{cor}[prop]{Corollary}
\newtheorem{teo}[prop]{Theorem}

\theoremstyle{definition}
\newtheorem{Def}[prop]{Definition}
\newtheorem*{Def*}{Definition}
\newtheorem*{prop*}{Proposition}
\newtheorem*{lema*}{Lemma}

\newtheorem{obs}[prop]{Remark}
\newtheorem*{obs*}{Remark}
\newtheorem*{cor*}{Corollary}
\newtheorem*{teo*}{Theorem}

\newtheorem*{es*}{Example}

\newtheorem*{main-teo}{Main Theorem}

\DeclareMathOperator{\Ker}{ker}

\DeclareMathOperator{\Hom}{Hom}

\DeclareMathOperator{\GL}{GL}

\DeclareMathOperator{\Sing}{Sing}

\DeclareMathOperator{\SL}{SL}

\DeclareMathOperator{\codim}{codim}

\DeclareMathOperator{\Gr}{Gr}

\DeclareMathOperator{\expdim}{expdim}

\DeclareMathOperator{\Terr}{Terr}
\DeclareMathOperator{\diam}{diam}

\definecolor{coloreteorema}{HTML}{E8EDFA}
    {\endMakeFramed}

    {\endMakeFramed}
    
\numberwithin{equation}{subsection}

\title{Identifiability and singular locus\\ of secant varieties to Grassmannians}
\author{Vincenzo Galgano\footnote{Dipartimento di Matematica, Università di Trento, Via Sommarive 14 38123 Trento, Italy; ORCID: 0000-0001-8778-575X (vincenzo.galgano@unitn.it)} , Reynaldo Staffolani\footnote{ORCID: 0000-0001-9889-5356 (reynaldo.staffolani@gmail.com)}}
\date{}

\begin{document}
	
	\maketitle
	
\begin{abstract}
	\noindent Secant varieties are among the main protagonists in tensor decomposition, whose study involves both pure and applied mathematic areas. Grassmannians are the building blocks for skewsymmetric tensors. Although they are ubiquitous in the literature, the geometry of their secant varieties is not completely understood. In this work we determine the singular locus of the secant variety of lines to a Grassmannian $Gr(k,V)$ using its structure as $\SL(V)$-variety. We solve the problems of identifiability and tangential-identifiability of points in the secant variety: as a consequence, we also determine the second Terracini locus to a Grassmannian. \\
	\hfill\break
	\hfill\break
	\end{abstract}
	
	\tableofcontents

\section*{Acknowledgement}
The authors warmly thank Alessandra Bernardi and Giorgio Ottaviani for the several and helpful discussions. Sincere thanks are due also to Laurent Manivel and Mateusz Micha\l ek for their useful suggestions, as well as to the anonymous referees for the interesting and detailed remarks which allowed us to improve our work. The authors are members of the Italian GNSAGA-INdAM.

\section*{Introduction}
\indent \indent Let $X \subset \mathbb{P}^N$ be a non-degenerate irreducible variety defined over an algebraically closed field $\mathbb{K}$ of characteristic $0$. Given a point $p \in \mathbb{P}^N$, its {\it $X$-rank} is the minimum number of points $p_1,\dots,p_r$ of $X$ such that $p \in \langle p_1,\dots,p_r \rangle$. The {\it $r$-th secant variety} of $X$, which we denote with $\sigma_r(X)$, is the Zariski closure of the union of all the points of $X$-rank at most equal to $r$. Such varieties are classically known and they are very versatile objects due to their deep connection with the {\it decomposition of tensors}. We refer to \cite{landsberg, bernardi2018hitchhiker, ottaviani2020tensor} for a general overview on secant varieties and tensors, and \cite{bernardi2012algebraic, comon2014tensors, lim2021tensors} for applications of the study of secant varieties and tensors. We also refer to \cite{bernardi2018hitchhiker} for a state of the art on the dimension of such varieties, and to \cite{galuppi2022secant,baur2007secant,laface2013secant} for some more detailed results. \\
\indent A particular aspect which is still unknown in general is the {\it singular locus} of such secant varieties. Classically, if $\sigma_r(X)$ is not a linear space, it is known that $\sigma_{r-1}(X) \subset \Sing(\sigma_r(X))$ but only in few cases $\Sing(\sigma_r(X))$ is actually determined.  For instance, the case $r=2$ for Segre varieties is solved in \cite{michalek2015secant}; for Veronese varieties \cite{kanev1999chordal} and \cite{kangjin2018} solve the cases $r=2$ and $r=3$ respectively, while partial results for higher cases $r\geq 4$ are obtained in \cite{kangjin2021}. The aim of this paper is to determine the singular locus of the $2$-nd secant variety (i.e. the secant variety of lines) to a Grassmannian $\Gr(k,V) \subset \mathbb{P}\left(\bigwedge^k V\right)$, the latter being the projective variety of $k$-dimensional linear subspaces of a vector space $V$. Partial results on this topic were obtained in \cite{manivelmichalek}. From now on we simply write {\em secant variety} by omitting \textquotedblleft$2$-nd\textquotedblright. Moreover, all along the paper we refer to the $X$-rank simply as the {\em skew-symmetric rank}, and from the duality of the Grassmannians we assume that $k \leq \frac{N}{2}$. Since the case $k=2$ (i.e. skew-symmetric matrices) is well-known and it is the only defective case for $\sigma_2(\Gr(k, V))$, we treat it separately (see Section \ref{sec:defective}) and in the following we also assume $k\geq3$.   \\
\hfill\break
\indent The main idea, which will be described more in detail along the paper, is to use the natural action of $\SL(V)$ on $\bigwedge^k V$. Indeed, the variety $\sigma_2(\Gr(k,V))$ can be written as a union of finitely many $\SL(V)$-orbits. Since $\Sing(\sigma_2(\Gr(k,V)))$ is stable under the action of $\SL(V)$, we must have that the singular locus is given by the union of some of these orbits contained in $\sigma_2(\Gr(k,V))$. In particular, one can check the singularity or smoothness of an orbit by simply checking it for a representative.\\
\indent In this respect, we use the following notion of distance (see Definition \ref{def:hamming distance} - cf.  \cite{catalisano2005secant,baur2007secant,abo2009non}): given $X \subset \mathbb P^n$ a projective variety and two points $p,q\in X$, the {\em Hamming distance} between $p$ and $q$ is the minimum number of lines of $\mathbb P^n$ contained in $X$ and connecting $p$ to $q$. This notion allows to distinguish the $\SL(V)$-orbits in the secant variety. \\
In Theorem \ref{thm:orbit stratification} we prove the following orbit stratification of the secant variety $\sigma_2(\Gr(k,V))$, as $l$ runs among the possible Hamming distances between points in $\Gr(k,V)$: in particular, the orbits  $\Sigma_{l,\Gr(k,V)}$ (defined in \eqref{def secant orbit}) lie in the open set $\sigma_2^\circ(\Gr(k,V))$, while the orbits $\Theta_{l,\Gr(k,V)}$ (defined in \eqref{def tangent orbit}), lie in the tangential variety $\tau(\Gr(k,V))$. The arrows denote the inclusion of an orbit into the closure of the other one.
	\begin{center}
		{\small 
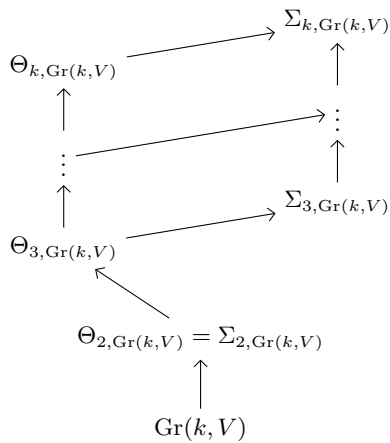
\begin{tikzpicture}[scale=3]
		
		\node(S) at (0,0.2){{$\Gr(k,V)$}};
		\node(t2) at (0,0.6){{$\Theta_{2,\Gr(k,V)} = \Sigma_{2,\Gr(k,V)}$}};
		\node(t3) at (-0.6,1){{$\Theta_{3,\Gr(k,V)}$}};
		\node(t) at (-0.6,1.4){{$\vdots$}};
		\node(td) at (-0.6,1.8){$\Theta_{k,\Gr(k,V)}$};
		
		\node(s3) at (0.6,1.2){{$\Sigma_{3,\Gr(k,V)}$}};
		\node(s) at (0.6,1.6){{$\vdots$}};
		\node(sd) at (0.6,2){{$\Sigma_{k,\Gr(k,V)}$}};
		
		\path[font=\scriptsize,>= angle 90]
		(S) edge [->] node [left] {} (t2)
		(t2) edge [->] node [left] {} (t3)
		(t3) edge [->] node [left] {} (s3)
		(t3) edge [->] node [left] {} (t)
		(t) edge [->] node [left] {} (td)
		(t) edge [->] node [left] {} (s)
		(td) edge [->] node [left] {} (sd)
		(s3) edge [->] node [left] {} (s)
		(s) edge [->] node [left] {} (sd);
		\end{tikzpicture}}
	\end{center}
 We completely determine such orbits, by exhibiting representatives and inclusions of their closures, and computing their dimensions (see Proposition \ref{prop:secant dimensions}, Proposition \ref{prop:tangent dimensions}). For what concerns the computation of the orbit dimensions in $\sigma_2(\Gr(k,V))$, we point out the paper \cite{BidlemanOeding}, that has appeared only recently and in which the authors compute the dimensions of the restricted chordal varieties $\sigma_2^l(\Gr(k,V))$ (corresponding to the closures of our {\em secant orbits} $\Sigma_{l,\Gr(k,V)}$) using similar techniques to the ones we have used.\\
\indent %Apolarity Theory, which is a useful algebraic tool for tensor decomposition, turned out to be handy also in our case. In particular, we use both {\em skew-symmetric apolarity} \cite{ABMM21}, and {\em nonabelian apolarity} \cite{landsberg2013equations}, also for 
We also solve the problems of {\em identifiability} and {\em tangential-identifiability} (see Definition \ref{def:cactus}) for points in $\sigma_2(\Gr(k,V))$. More precisely, in the above notations, we prove the following result in Theorem \ref{thm:identifiability for secant orbits} and Theorem \ref{thm:tangential-identifiability for tangent orbits}. 
\begin{teo*}
	The orbits $\Sigma_{l,\Gr(k,V)}$ for $l\geq 3$ are the only ones in $\sigma_2(\Gr(k,V))$ being identifiable. The orbits $\Theta_{l,\Gr(k,V)}$ for $l\geq 3$ are the only ones in $\tau(\Gr(k,V))$ being tangential-identifiable.
	\end{teo*}
\noindent As a consequence, in Theorem \ref{thm:terracini locus} we prove that the second Terracini locus (introduced in \cite{ballico2020terracini,ballico2021terracini}) of the Grassmannian $\Gr(k,V)$ corresponds to the distance-$2$ orbit closure $\overline{\Sigma_{2,\Gr(k,V)}}$. The main result of this work is the following (see Theorem \ref{thm:singular locus}).
\begin{main-teo}
	For any $3 \leq k \leq \frac{\dim V}{2}$ and $\dim V\geq 7$, the singular locus of the secant variety $\sigma_2(\Gr(k,V))$ is the closure of the distance-$2$ orbit: \[\Sing(\sigma_2(\Gr(k,V)))=\overline{\Sigma_{2,\Gr(k,V)}} \ . \]
	\end{main-teo}

\noindent We prove it by computing the dimension of the tangent space at a representative of each orbit, also applying the results we obtained about identifiability. In particular, our result corrects a previous statement in \cite[before Figure 1]{abo2009non} in which the authors stated that $\Sing(\sigma_2(\Gr(3,\mathbb C^7))) = \Gr(3,\mathbb C^7)$.\\
\indent Quite remarkably, the same results of this work hold for the case of spinor varieties: they have been proved in \cite{galgano2023spinor} by the first author simultaneously with this work. This allows us to emphasize the importance of our work, which proposes a new perspective for studying a wider class of varieties. In this respect, we hope that this is the first of a series of papers generalizing our results.\\
\hfill\break
\indent The paper is organized as follows. In Section \ref{sec:preliminaries} we recall some basic facts. In Section \ref{sec:defective} we discuss the toy case $k=2$ in order to give a hint of the general idea behind our construction. In Section \ref{sec:orbit partition} we classify the orbits in $\sigma_2(\Gr(k,V))$, distinguishing them in secant orbits and tangent orbits: we also show the inclusions among their closures and compute the dimensions of the secant orbits. In Section \ref{sec:identifiability of secant orbits} we determine which points are identifiable in $\sigma_2(\Gr(k,V))$, while in Section \ref{sec:cactus} we give the definition of tangential-identifiability in $\tau(\Gr(k,V))$ and we solve the problem of determining which tangent points are tangential-identifiable: as a consequence, we also determine the dimensions of the tangent orbits. These studies allow us to determine the second Terracini locus $\Terr_2(\Gr(k,V))$ too, which is the topic of Section \ref{sec:terracini locus}. The previous sections are also preparatory for the computation of the singular locus in Section \ref{sec:singular locus}. 

\section{Preliminaries}\label{sec:preliminaries}

\indent \indent Let $V$ be a finite dimensional vector space over the complex field $\mathbb C$: actually, all arguments and results hold over any algebraically closed field of characteristic $0$. We assume to have fixed a basis of $V$ so that $V \simeq \mathbb C^N$. We fix the notation $i=1:s$ meaning $i=1,\ldots,s$.\\ 
\indent First, we recall some notions in the theory of Tensor Decomposition. 

\paragraph{Ranks and secant varieties.} Given $X \subset \mathbb P^M$ an irreducible non-degenerate projective variety, the $X$-{\it rank} of a point $p \in \mathbb P^M$, denoted by $r_X(p)$, is the minimum number of points of $X$ whose span contains $p$. The {\it $r$-th secant variety of $X$} is defined as the Zariski closure
\[\sigma_r(X) := \overline{\{p \in \mathbb P^M: r_X(p) \leq r\}} \subset \mathbb P^M \ . \]
The {\it border $X$-rank} of a point $p \in \mathbb P^M$, denoted by $br_X(p)$, is the minimum integer $r$ such that $p \in \sigma_r (X)$. In the case of $X$ being a Grassmannian in its Pl\"{u}cker embedding, we refer to the (border) $X$-rank as {\it skew-symmetric} (border) rank. Although computing the dimension of secant varieties is a hard problem in general, the following inequality always holds
\[\dim \sigma_r (X) \leq \expdim \sigma_r (X):=\min \{r(\dim(X)+1)-1, M\} \ , \]
where the right-hand side is called {\it expected dimension} of the secant variety.

\paragraph{Identifiability.} For any $p \in \mathbb P^M$ with $r_X(p) = r$, the {\it decomposition locus} of $p$ is the set of all points of $X$ giving a decomposition of $p$
\[ Dec_X(p) := \left\{(p_1,\dots,p_r) \ | \ p_i \in X,\ p \in \langle p_1,\dots,p_r \rangle \right\} \subset X^r_{/\mathfrak S_r} \ , \]
where $\mathfrak S_r$ denotes the symmetric group acting on $r$ elements. An element $(p_1,\dots,p_r) \in Dec_X(p)$ is called a {\it decomposition} of $p$. A point $p$ is {\it identifiable} if there exists a unique decomposition of $p$, i.e. $Dec_X(p)$ is a singleton. Otherwise one says that $p$ is {\it unidentifiable}.\\
For any subset $Y\subset X$, we say that $Y$ is (un)identifiable if any point of $Y$ is so. In this case we say that the Zariski closure $\overline{Y}$ is {\it generically (un)identifiable}. In particular, an orbit is (un)identifiable if and only if a representative of it is so.

\paragraph{Tangential-identifiability} The {\em tangential variety} to $X$ in $\mathbb P^M$ is the union of all lines in $\mathbb P^M$ which are tangent to $X$, i.e. it is the set of all tangent points to $X$
\[ \tau(X):=\bigcup_{p\in X}T_pX \ \subset \mathbb P^M \ . \]
In particular, $X\subset \tau(X)\subset \sigma_2(X)$. 

\begin{Def}\label{def:cactus}
	A tangent point $q \in \tau(X)$ is {\em tangential-identifiable} if it lies on a unique tangent line to $X$, or equivalently if there exists a unique $p \in X$ such that $q \in T_pX$. Otherwise it is {\em tangential-unidentifiable}.
\end{Def}

\noindent We say that an orbit is tangential-(un)identifiable if all of its elements are so.

\begin{Def}\label{def:tangential-locus}
	Given a tangent point $q \in \tau(X)$, its {\em tangential-locus} is the set of points $p \in X$ such that $q \in T_{p}X$.
\end{Def}

\noindent Clearly, if $q$ is tangential-identifiable, then its tangential-locus is given by a single point in $X$.

\paragraph{The $2$-nd secant variety of Grassmannians.} Let $\Gr(k,\mathbb C^N)$ be the projective variety of $k$-dimensional linear subspaces in $\mathbb C^N$. We want to study the secant variety $\sigma_2(\Gr(k,\mathbb C^N))$ as an $\SL(N)$-variety. By duality we may assume that $k \leq \frac{N}{2}$.		
At first we observe that the action
\[\SL(N) \times \bigwedge^k \mathbb C^N \longrightarrow \bigwedge^k \mathbb C^N\]
preserves the skew-symmetric (border) rank of a tensor. Therefore regarding the variety $\sigma_2(\Gr(k,\mathbb C^N))$ as the space of skew-symmetric tensors of border rank at most $2$, we get the restricted action
\[\SL(N) \times \sigma_2(\Gr(k,\mathbb C^N)) \longrightarrow \sigma_2(\Gr(k,\mathbb C^N)) \ . \] 
\indent Recall that $\dim \SL(N) = N^2-1$ and $\dim \sigma_2 (\Gr(k,\mathbb C^N)) \leq \expdim \sigma_2 (\Gr(k,\mathbb C^N)) = 2k(N-k)+1$. Thus the assumption $k \leq \frac{N}{2}$ leads to the inequality
\[ N^2-1 = (N-1)(N+1) \geq (N-k)(2k+1) = 2k(N-k) + N-k \geq 2k(N-k)+1 \ , \]
that is $\dim \SL(N) \geq \dim \sigma_2 (\Gr(k,\mathbb C^N))$, which is a necessary condition for having finitely many orbits. Along this section we are going to describe such orbits, together with the dimensions and inclusions of their closures.\\
\indent The main idea is to distinguish the study in two branches, namely the {\em secant} branch and the {\em tangent} branch: such a distinction comes from the definition of secant variety to a variety that, roughly speaking, is the union of bisecant lines and tangent lines to the variety. 

\section{The defective toy-case $\Gr(2,\mathbb C^N) \subset \mathbb P(\bigwedge^2\mathbb C^N)$}\label{sec:defective}

\indent \indent The toy case $k=2$, in which $\bigwedge^2\mathbb C^N$ parametrizes the space $N \times N$ skew-symmetric matrices, is the only defective case (regarding the secant variety of lines) \cite[Theorem 1.1]{boralevi2013note} and it has to be considered on its own. As $\sigma_2(\Gr(2,\mathbb C^N))$ parametrizes the $N \times N$ skew-symmetric matrices of rank at most $4$ (i.e. skew-symmetric rank at most $2$), one easily gets that $\sigma_2(\Gr(2,\mathbb C^4))=\mathbb P^{5}$ and $\sigma_2(\Gr(2,\mathbb C^5))=\mathbb P^{9}$. In this respect, we assume $N\geq 6$. We recall that $\sigma_2(\Gr(2,\mathbb C^N))$ is defective with defect equal to $4$, that is $\dim \sigma_2(\Gr(2,\mathbb C^N))=(2\dim \Gr(2,\mathbb C^N)+1)-4=4(N-2)-3$. Moreover, from the key result \cite[Theorem 1.4]{zak1993tangents} we have that $\tau(\Gr(2,\mathbb C^N)) = \sigma_2(\Gr(2,\mathbb C^N))$. 

\begin{teo}
For $N\geq 6$ the singular locus $\Sing(\sigma_2(\Gr(2,\mathbb C^N)))$ is exactly $\Gr(2,\mathbb C^N)$.
\end{teo}
\begin{proof}
Since $\Gr(2,\mathbb C^N)$ parametrizes the $N \times N$ skew-symmetric matrices of rank $2$ (i.e. skew-symmetric rank $1$), the variety $\sigma_2(\Gr(2,\mathbb C^N))$ is given by the union of only two orbits: $\Gr(2,\mathbb C^N)$ and $\sigma_2(\Gr(2,\mathbb C^N)) \setminus \Gr(2,\mathbb C^N)$. Since in general $\Sing(\sigma_r(\Gr(k,\mathbb C^N))) \supset \sigma_{r-1}(\Gr(k,\mathbb C^N))$ holds, the thesis follows.
\end{proof}

\section{$\SL(N)$-orbit structure of $\sigma_2(\Gr(k,\mathbb C^N))$}\label{sec:orbit partition}

\indent \indent We describe the $SL(N)$-orbit partition of $\sigma_2(\Gr(k,\mathbb C^N))$ for $k\geq 3$. Recall that by duality we assume $k \leq \frac{N}{2}$. The following notion was already introduced in \cite{catalisano2005secant,baur2007secant,abo2009non}.

\begin{Def}\label{def:hamming distance}
The {\em Hamming distance} $d(p,q)$ between two distinct points $p,q \in \Gr(k,\mathbb C^N)$ is the minimum number of lines contained in $\Gr(k,\mathbb C^N)$ and connecting $p$ to $q$, that is  
\[d(p,q) := \min \left\{ r \ | \ \exists p_1,\dots,p_{r-1} \in \Gr(k,\mathbb C^N) \ \text{s.t.} \ \langle p_i,p_{i+1}\rangle\subset \Gr(k,\mathbb C^N), \ i = 0:r-1 \right\}\]
where $p_0=p$ and $p_r=q$.
\end{Def}

\noindent In particular, $d(p,q)=1$ if and only if $\langle p,q\rangle \subset \Gr(k,\mathbb C^N)$. By convention, we say that $d(p,q)=0$ if and only if $p=q$. The Hamming distance between two points in $\Gr(k,N)$ can be reinterpreted in terms of intersection of subspaces. 

\begin{lemma}\label{distance as intersection}
	Let $p,q \in \Gr(k,\mathbb C^N)$ and let $H_p,H_q\subset \mathbb C^N$ be their corresponding $k$-dimensional subspaces. Then $p,q$ are joined by a line in the Grassmannian if and only if the corresponding subspaces $H_p,H_q$ meet along a common hyperplane. In particular,
	\[ d(p,q)=\codim_{H_p}(H_p\cap H_q)=k-\dim(H_p\cap H_q) \ .\]
	\end{lemma}

\begin{proof}
	The trivial case $p=q$ (i.e. $H_p=H_q$) implies $d(p,q)=0=k-\dim(H_p\cap H_q)$. If $d(p,q)=1$, then $\langle p,q \rangle \subset \Gr(k,\mathbb C^N)$ and $p+q\in \Gr(k,\mathbb C^N)$, which is equivalent to $\dim(H_p\cap H_q)=k-1$. \\
	\indent For the general case, assume $d(p,q)=\ell\geq 2$ and let $\dim(H_p\cap H_q)=s$. First, we prove that $\ell \leq k-s$. We may assume $p=x_1\wedge \ldots x_s \wedge \ldots \wedge x_k$ and $q=x_1\wedge \ldots \wedge x_s\wedge y_{s+1}\wedge \ldots \wedge y_k$. Consider the points $p_1,\ldots, p_{k-s-1}\in \Gr(k,\mathbb C^N)$ corresponding to the $k$-dimensional subspaces 
		\[
			H_{p_j} = \langle x_1,\ldots , x_s, \ldots , x_{k-j} , y_{k-j+1},\ldots , y_k\rangle \ \ \ \ , \ \forall j=1:k-s-1 \ .
			\]
		Since $\dim(H_{p_j}\cap H_{p_{j+1}})=k-1$ for any $j=0:k-s$ (where $p_0=p$ and $p_{k-s}=q$), it follows that $d(p_{j},p_{j+1})=1$ for any $j=0:k-s$, that is $d(p,q)=\ell \leq k-s$.\\
	\indent On the other hand, from Definition \ref{def:hamming distance} there exists $c_1,\dots,c_{\ell-1} \in \Gr(k,\mathbb C^N)$ such that $\langle c_i, c_{i+1} \rangle \subset \Gr(k, \mathbb C^N)$ for any $i=0:\ell-1$, where $c_0 = p$ and $c_\ell = q$. Clearly it holds $d(p,c_{\ell-1}) = \ell-1$, hence by induction one has $\ell-1 = d(p,c_{\ell-1}) = k - \dim(H_p \cap H_{c_{\ell-1}})$, that is $\dim(H_p \cap H_{c_{\ell-1}}) = k-\ell+1$. Now consider the intersection $H_p \cap H_{c_{\ell-1}} \cap H_q$: using Grassmann Formula its dimension is	
	\begin{align*} &\dim(H_p \cap H_{c_{\ell-1}} \cap H_q) = \\
	&= \dim(H_p \cap H_{c_{\ell-1}}) + \dim(H_{c_{\ell-1}} \cap H_q) - \dim\big((H_p \cap H_{c_{\ell-1}})+(H_{c_{\ell-1}} \cap H_q)\big) \\
	&= 2k-\ell -\dim((H_p \cap H_{c_{\ell-1}})+(H_{c_{\ell-1}} \cap H_q)) \ .
	\end{align*}
	Since the latter dimension can be either $k-1$ or $k$, one has
	\[ k-\ell \leq \dim(H_p \cap H_{p_{\ell-1}} \cap H_q) \leq k-\ell +1 \ . \]
	If it was $\dim(H_p \cap H_{c_{\ell-1}} \cap H_q)=k-\ell +1$, from the inequality $\ell \leq k-s$ one would get 
	\[ k-\ell+1 = \dim (H_p \cap H_{c_{l-1}}\cap H_q) \leq \dim(H_p\cap H_q) = s \leq k-\ell \ , \]
	leading to contradiction. It follows that $\dim(H_p \cap H_{p_{\ell-1}} \cap H_q)=k-\ell$ must hold and, from a similar chain of inequalities as above, we conclude that $\ell = k-s$, that is the thesis.
	\end{proof}

\begin{Def}
	The {\em diameter} of the Grassmannian $\Gr(k,\mathbb C^N)$ is the maximum possible distance between its points, i.e. $\diam\Gr(k,\mathbb C^N)=\max\{d(p,q) \ | \ p,q \in \Gr(k,\mathbb C^N)\}$.
	\end{Def}

\begin{cor}
	The Grassmannian $\Gr(k,\mathbb C^N)$ has diameter $\diam\Gr(k,\mathbb C^N)=k$.
	\end{cor}

We will refer to points of $\sigma_2(\Gr(k,\mathbb C^N))$ lying on bisecant lines as {\em secant points}, and to points lying on tangent lines as {\em tangent points}. Analogously, we will consider {\em secant orbits} and {\em tangent orbits}. We stress out that there could be points being both secant and tangent (see Remark \ref{rmk on tangent orbits}).

\subsection{The secant branch} 
\indent \indent Let $(e_1,\ldots, e_N)$ be the standard basis of $\mathbb C^N$ and $k\geq 3$. Fix the notations for $l=1:k-1$
\begin{align}\label{representatives secant}
	\omega & := e_1\wedge \ldots \wedge e_k\\
	\mathbbm{e}_l & := e_1\wedge \ldots \wedge e_{k-l} \wedge e_{k+1} \wedge \ldots \wedge e_{k+l}\\
	s_l & := \omega + \mathbbm{e}_l = e_1\wedge \ldots \wedge e_{k-l} \wedge (e_{k-l+1}\wedge \ldots \wedge e_k + e_{k+1} \wedge \ldots \wedge e_{k+l})
	\end{align}
and
\begin{align*}
	\mathbbm{e}_k & :=e_{k+1}\wedge \ldots \wedge e_{2k}\\
	s_k & := \omega + \mathbbm{e}_k = e_1\wedge \ldots \wedge e_k  + e_{k+1} \wedge \ldots \wedge e_{2k} \ .
	\end{align*}
Notice that $d(\omega,\mathbbm{e}_l)=l$. We show that the {\em secant} $\SL$-orbits in $\sigma_2(\Gr(k,\mathbb C^N))$ are as many as the diameter of $\Gr(k,\mathbb C^N)$, and the points $s_l$ are their representatives. For any $l=1:k$, we denote 
\begin{equation}\label{def secant orbit}
	\Sigma_{l,\Gr(k,\mathbb C^N)} := \left\{p+q \in \sigma_2(\Gr(k,\mathbb C^N)) \ \big| \ d(p,q)=l\right\} \ .
	\end{equation}
Notice that the $\SL$-action preserves the dimensions $\dim(H_p\cap H_q)$ for any $p,q \in \Gr(k,\mathbb C^N)$, and by Lemma \ref{distance as intersection} it preserves the Hamming distance between points in $\Gr(k,\mathbb C^N)$: thus the action of $\SL(N)$ preserves $\Sigma_{l,\Gr(k,\mathbb C^N)}$. 

\begin{prop}\label{prop:partition of secant branch}
	For any $l=1:k$, the set $\Sigma_{l,\Gr(k,\mathbb C^N)}$ is a $\SL$-orbit with representative $s_l$: 
	\[\Sigma_{l,\Gr(k,\mathbb C^N)} = \SL(N)\cdot s_l \ .\]
	In particular, the $\SL$-orbit partition of the open subset $\sigma_2^\circ(\Gr(k,\mathbb C^N))$ of points lying on secant lines is $\sigma_2^\circ(\Gr(k,\mathbb C^N))= \bigcup_{l=1}^k\Sigma_{l,\Gr(k,\mathbb C^N)}$.
	\end{prop}
\begin{proof}
	Clearly, the orbit $\SL(N)\cdot s_l$ is contained in $\Sigma_{l,\Gr(k,\mathbb C^N)}$ but actually equality holds: given $p+q \in \Sigma_{l,\Gr(k,\mathbb C^N)}$, we can write it as
	\[p+q = v_1 \wedge \dots \wedge v_{k-l} \wedge \left(v_{k-l+1} \wedge \dots \wedge v_k + v_{k+1} \wedge \dots \wedge v_{k+l} \right )\]
	and one can always find a $g \in SL(N)$ such that $g(v_i) = e_i$ for any $i = 1:k+l$, that is $g\cdot (p+q)=s_l$. 
	\end{proof}

\begin{obs}\label{rmk on secant orbits}
	By the theory of linear algebraic groups, the orbit $\Sigma_{1,\Gr(k,\mathbb C^N)}=\Gr(k,\mathbb C^N)$ is the unique {\em closed} $\SL$-orbit in $\mathbb P(\bigwedge^k\mathbb C^N)$. Moreover, the orbit $\Sigma_{k,\Gr(k,\mathbb C^N)}$ is {\em dense}: indeed, another representative of it is given by $\omega + e_{N-k+1}\wedge \ldots \wedge e_N$, that is the sum of the highest and lowest weight vectors (respectively) in the irreducible $\SL(N)$-representation $\bigwedge^k\mathbb C^N$. Finally, the closures $\overline{\Sigma_{l,\Gr(k,\mathbb C^N)}}$ are already known in the literature as {\em restricted chordal varieties} \cite[Hard exercise 15.44]{FH91}.
	\end{obs}

\indent In the following we reinterpret the vector subspaces corresponding to Grassmannian points as kernels and we associate certain vector subspaces to secant points too: we underline that the latter is {\em not} a 1:1 correspondence, in the sense that to any vector subspace could correspond more secant points. 
%For $p=v_1\wedge \ldots \wedge v_k \in \Gr(k,\mathbb C^N)$, the multiplication map
%\[ \begin{matrix}
%\psi_p: & \mathbb C^N & \longrightarrow & \bigwedge^{k+1}\mathbb C^N \\
%& x & \mapsto & x \wedge p
%\end{matrix}\]
%has kernel $\Ker(\psi_p)=H_p=\langle v_1,\ldots, v_k\rangle_{\mathbb C}$. In a similar way, 
For any point $q \in \bigwedge^k\mathbb C^N$, we consider the multiplication map
\begin{equation}\label{multiplication map}
	\begin{matrix}
	\psi_q: & \mathbb C^N & \longrightarrow & \bigwedge^{k+1}\mathbb C^N \\
	& x & \mapsto & x \wedge q
	\end{matrix} \end{equation}
and we associate to the point $q$ the subspace
\[ H_{q}:=\Ker(\psi_q) \ . \] 
For instance, for $q=v_1\wedge \ldots \wedge v_k \in \Gr(k,\mathbb C^N)$ one recovers the corresponding subspace $\Ker(\psi_p)=\langle v_1,\ldots, v_k\rangle_{\mathbb C}=H_p$. Notice that at the moment we know the dimension and a basis of $H_q$ only for $q\in \Gr(k,\mathbb C^N)$.

\begin{lemma}\label{lemma:intersection subspaces to generic secant pts}
	Let $p+q \in \Sigma_{k,\Gr(k,\mathbb C^N)}$ be a generic secant point. Then $H_{p+q}=\{0\}$.
	\end{lemma}
\begin{proof}
	By homogeneity of the dense orbit $\Sigma_{k,\Gr(k,\mathbb C^N)}$, we may assume $p+q=e_1\wedge \ldots \wedge e_k + e_{k+1}\wedge \ldots \wedge e_{2k}$. Notice that $\{0\} =H_p\cap H_q \subset H_{p+q}=\Ker(\psi_{p+q})$.\\
	Let $y \in H_{p+q}\subset \mathbb C^N$ with $y=\sum_{i=1}^N \beta_i e_i$: then
	{\small \begin{align*}
		0 & = y \wedge \left( e_1\wedge \ldots \wedge e_k + e_{k+1}\wedge \ldots \wedge e_{2k} \right)\\
		& = \sum_{i=k+1}^N(-1)^k\beta_i e_1\wedge \ldots \wedge e_k \wedge e_i \ + \ \sum_{i=1}^{k}\beta_i e_i\wedge e_{k+1}\wedge \ldots \wedge e_{2k} \ +  \sum_{i=2k+1}^{N}(-1)^k\beta_i e_{k+1}\wedge \ldots \wedge e_{2k}\wedge e_i \ .	
		\end{align*}}
	From the linear independence of the summands above in $\bigwedge^{k+1}\mathbb C^N$, it follows $\beta_i=0$ for any $i=1:N$, that is $y=0$ and the thesis follows.
	\end{proof}

\begin{prop}\label{prop:intersection subspaces to secant pts}
	Let $2\leq l\leq k$ and let $p+q\in \Sigma_{l,\Gr(k,\mathbb C^N)}$. Then $H_{p+q}=H_p\cap H_q$.
	\end{prop}
\begin{proof}
	By Lemma \ref{lemma:intersection subspaces to generic secant pts} we know that the thesis holds for $l=k$. Fix $2\leq l\leq k-1$ and a point $p+q\in \Sigma_{l,\Gr(k,\mathbb C^N)}$. Let $p=v_1\wedge \ldots \wedge v_{k}$ and $q=v_1\wedge \ldots \wedge v_{k-l}\wedge w_{k-l+1}\wedge \ldots \wedge w_k$, so that $H_{p}\cap H_q=\langle v_1, \ldots , v_{k-l}\rangle_{\mathbb C}$.\\
	Consider the multiplication map $\psi_{p+q}$ as in \eqref{multiplication map}: then it clearly holds $H_p\cap H_q \subset H_{p+q}:=\Ker(\psi_{p+q})$. Take $y \in H_{p+q}$ being linearly independent from $v_1,\ldots, v_{k-l}$: in particular, if we complete $\{v_1,\ldots , v_{k-l}\}$ to a basis of $\mathbb C^N$, we may assume that $y$ does not depend on $v_1,\ldots, v_{k-l}$. Then we get
	\begin{align*}
		0 & = y \wedge (p+q) = y \wedge v_1\wedge \ldots \wedge v_{k-l}\wedge \left( v_{k-l+1}\wedge \ldots \wedge v_k + w_{k-l+1}\wedge \ldots \wedge w_k \right)\\
		& = (-1)^{k-l} v_1\wedge \ldots \wedge v_{k-l}\wedge \left( y\wedge\underbrace{ v_{k-l+1}\wedge \ldots \wedge v_k}_{=: \ a} \ + \ y\wedge \underbrace{w_{k-l+1}\wedge \ldots \wedge w_k}_{=: \ b} \right) \ .
		\end{align*}
%	\textcolor{red}{Since $y\wedge(a+b)$ is linearly independent from $v_1,\ldots , v_{k-l}$, it follows that $y \wedge (a+b)=0$. Notice that, if we denote $V':=\mathbb C^N/H_p\cap H_q\simeq \mathbb C^{N-k+l}$, we have $a,b \in \Gr(l,V')$ with $d(a,b)=l$, thus $a+b$ lies in the dense orbit $\Sigma_{l, \Gr(l,V')}$.} \\
	If we denote $V':=\mathbb C^N/H_p\cap H_q\simeq \mathbb C^{N-k+l}$, then the multiplication map $v_1\wedge \ldots \wedge v_{k-l}\wedge \bullet : \bigwedge^{l+1}V' \rightarrow \bigwedge^{k+1}\mathbb C^N$ is injective, hence the above equality implies $y\wedge(a+b)=0$. Notice also that $a,b \in \Gr(l,V')$ with $d(a,b)=l$, thus $a+b$ lies in the dense orbit $\Sigma_{l, \Gr(l,V')}$. 
	Moreover, $y \in \Ker\left( \psi_{a+b}: V'\rightarrow \bigwedge^{l+1}V'\right)=: H_{a+b}$ and by Lemma \ref{lemma:intersection subspaces to generic secant pts} it holds $H_{a+b}=\{0\}$. It follows that $y=0$, thus $H_{p+q}=H_p\cap H_q$.
	\end{proof}

\subsection{The tangent branch} 
\indent \indent Now we focus on orbits of points lying on tangent lines to $\Gr(k,\mathbb C^N)$, whose union gives the tangential variety $\tau(\Gr(k,\mathbb C^N))$. By homogeneity of the Grassmannian, it is enough to study only one tangent space. Fix the standard basis $\{e_1,\ldots, e_N\}$ of $\mathbb C^N$ and consider the point $[\omega]=[e_1 \wedge \dots \wedge e_k ] \in \Gr(k,\mathbb C^N)$. It is known that the (affine) tangent space at $[\omega]$ is
\begin{align}\label{SpTang:Grass}
	T_{[\omega]} \Gr(k, \mathbb C^N) &= \mathbb C^N \wedge e_2 \wedge \dots \wedge e_k + \dots + e_1 \wedge \dots \wedge e_{k-1} \wedge \mathbb C^N  \\
	&= \left\langle e_1 \wedge \dots \wedge e_k, \{ e_1\wedge \ldots \wedge \widehat{e_i} \wedge \ldots \wedge e_k \wedge e_l \ | \ i=1:k, l=k+1:N\} \right\rangle \nonumber
	\end{align}
where $\widehat{e_i}$ denotes that $e_i$ has been removed. 

\begin{obs} \label{Isom:TangSp&Mat}
	For a given $p=[v_1\wedge \ldots \wedge v_k] \in \Gr(k, \mathbb C^N)$ one has the following isomorphism \cite[Example 16.1]{harris2013algebraic} wich is compatible with the group action:
	\[T_p \Gr(k,\mathbb C^N) = \bigwedge^{k-1}H_p\wedge V \simeq \Hom \left(H_p, \mathbb C^N / H_p \right ) \ . \]
	\end{obs}

\begin{obs} \label{tangent space:dist}
	One can describe the tangent space to the Grassmannian at some point $p$ using the notion of Hamming distance. If $p = v_1 \wedge \dots \wedge v_k \in \Gr(k,\mathbb C^N)$, then from \eqref{SpTang:Grass} one gets that $T_p \Gr(k,\mathbb C^N) = \left\langle \big\{q \in \Gr (k,\mathbb C^N) \ \big| \ d(p,q) \leq 1  \big\}\right\rangle$, that is the tangent space at a point is generated by lines passing through that point. 
	\end{obs}

%Let us now study the possible sums of generators in \eqref{SpTang:Grass}. Every generator is a point of $\Gr(k,\mathbb C^N)$, and the distance between $p$ and any other generator is $1$, hence their sum has rank $1$ too: the same holds for any two generators obtained by removing the same basis vector $e_i$. It remains to study sums of the generators in which every two addends are determined by removing two different vectors $e_i$ and $e_j$. 
Consider the $k$ elements in $T_\omega\Gr(k,\mathbb C^N)$
\begin{equation} \label{representatives tangent} 
	\theta_l:=\sum_{j=1}^l e_1 \wedge \dots \wedge e_{j-1} \wedge e_{k+j} \wedge e_{j+1} \wedge \dots \wedge e_k \ \ \ \ , \ \forall l=1:k \ .
	\end{equation}
\noindent From Remark \ref{Isom:TangSp&Mat}, any element of $T_\omega\Gr(k,\mathbb C^N)$ corresponds to an $(N-k)\times k$ matrix in $\mathbb C^{N-k}\otimes \mathbb C^{k}$: in particular, any $\theta_l$ corresponds to an $(N-k) \times k$ matrix of rank $l$. The only invariant in $\mathbb C^{N-k}\otimes \mathbb C^k$ is the rank and, since the isomorphism is compatible with the group action, so is for $T_\omega\Gr(k,\mathbb C^N)$. In particular, all points in $T_\omega\Gr(k,\mathbb C^N)$ of rank $l$ are conjugated to $\theta_l$.
%\indent Now, the union of the tangent spaces $T_p\Gr(k,\mathbb C^N)$, as $p \in \Gr(k,\mathbb C^N)$ varies, gives the tangential variety $\tau(\Gr(k,\mathbb C^N))$. By the previous arguments, the orbit partition of any tangent space $T_p\Gr(k,\mathbb C^N)$ is determined by the rank of $(N-k)\times k$ matrices. 
Finally, by homogeneity of $\Gr(k,\mathbb C^N)$, the action of $\SL(N)$ conjugates all tangent spaces, and for any $l=1:k$ the unions of all the rank-$l$ orbits (as the tangent space varies) gives an $\SL(N)$-orbit in the tangential variety $\tau(\Gr(k,\mathbb C^N))$, namely
\begin{equation}\label{def tangent orbit}
	\Theta_{l,\Gr(k,\mathbb C^N)}:= \left\{ t \in \tau(\Gr(k,\mathbb C^N)) \ | \ t \ \text{has rank} \ l\right\} \ .
	\end{equation}
From the arguments above we conclude the following result.

\begin{prop}\label{prop:tangent orbits}
	For any $l=1:k$, the set $\Theta_{l,\Gr(k,\mathbb C^N)}$ coincides with the $\SL$-orbit \[\Theta_{l,\Gr(k,\mathbb C^N)} = \SL(N)\cdot \theta_l \ . \]
	In particular, the $\SL$-orbit partition of the tangential variety is $\tau(\Gr(k,\mathbb C^N))= \bigcup_{l=1}^k\Theta_{l,\Gr(k,\mathbb C^N)}$.
\end{prop}

\begin{obs}\label{rmk on tangent orbits}
	From \eqref{representatives tangent} we have $\theta_1=e_2\wedge \ldots \wedge e_k \wedge e_{k+1}\in \Gr(k,\mathbb C^N)$. Together with Remark \ref{rmk on secant orbits} it follows 
	\[\Sigma_{1,\Gr(k,\mathbb C^N)}=\Theta_{1,\Gr(k,\mathbb C^N)}=\Gr(k,\mathbb C^N) \ .\]
	Morover, $\theta_2=e_2\wedge \ldots \wedge e_k\wedge e_{k+1} + e_1\wedge e_3\wedge \ldots \wedge e_k \wedge e_{k+2}$ is sum of two points in $\Gr(k,\mathbb C^N)$ having Hamming distance $2$, thus $\theta_2\in \Sigma_{2,\Gr(k,\mathbb C^N)}$ and we get
	\[ \Sigma_{2,\Gr(k,\mathbb C^N)}=\Theta_{2,\Gr(k,\mathbb C^N)} \ .\]
	Finally, the orbit $\Theta_{k,\Gr(k,\mathbb C^N)}$ is given by points corresponding to maximum rank $(N-k)\times k$ matrices, which are a dense subset in $\mathbb C^{N-k}\otimes \mathbb C^k$. Thus the orbit $\Theta_{k,\Gr(k,\mathbb C^N)}$ is {\em dense} in $\tau(\Gr(k,\mathbb C^N))$.
	\end{obs}

\subsection{Orbit closures: inclusions; dimensions of secant orbits} 

\indent \indent From the previous subsection we get the $\SL(N)$-orbit partition
\[ \sigma_2(\Gr(k,\mathbb C^N))=\Gr(k,\mathbb C^N) \sqcup \left( \bigsqcup_{l=2}^k \Sigma_{l,\Gr(k,\mathbb C^N)} \right) \cup \left( \bigsqcup_{l=3}^k\Theta_{l,\Gr(k,\mathbb C^N)} \right) \ , \]
where the non-disjoint union appears as we haven't proved yet that none of the tangent orbits $\Theta_{l,\Gr(k,\mathbb C^N)}$ for $l\geq 3$ coincides with a secant orbit: we prove this in Proposition \ref{prop:no collapsing}.

\paragraph{Orbit inclusions.} First, we determine the inclusions among the orbit closures. 

\begin{prop} \label{IncOrb:Sec}
	For any $l=1:k-1$ it holds $\Sigma_{l,\Gr(k,\mathbb C^N)} \subset \overline{\Sigma_{l+1,\Gr(k,\mathbb C^N)}}$.
	\end{prop}
\begin{proof}
	Fix $l=1:k-1$. By homogeneity, it is enough to show that a representative of the distance-$l$ orbit lies in the closure of the distance-$(l+1)$ orbit. The representative $s_l\in \Sigma_{l,\Gr(k,\mathbb C^N)}$ is limit for $\epsilon \rightarrow \infty$ of the sequence
	\[\omega + e_1\wedge \ldots \wedge e_{k-l-1}\wedge \left ( e_{k-l}+\frac{1}{\epsilon}e_{k+l+1} \right)\wedge e_{k+1}\wedge \ldots \wedge e_{k+l} \in \Sigma_{l+1,\Gr(k,\mathbb C^N)} \ . \] \vskip-0.5cm
	\end{proof}

\begin{prop}\label{IncOrb:Tang}
	For any $l=1:k-1$ it holds $\Theta_{l,\Gr(k,\mathbb C^N)}\subset \overline{\Theta_{l+1,\Gr(k,\mathbb C^N)}}$.
	\end{prop}
\begin{proof}
	The tangent points in $\Theta_{l,\Gr(k,\mathbb C^N)}$ correspond to $(N-k)\times k$ matrices of rank $l$, while points $\Theta_{l+1,\Gr(k,\mathbb C^N)}$ to $(N-k)\times k$ matrices of rank $l+1$. The thesis follows from the fact that the former matrices lie in the closure of the latter ones.
	\end{proof}

\begin{prop} \label{IncOrc:SecVsTang}
	For any $l=3:k$ it holds $\Theta_{l,\Gr(k,\mathbb C^N)} \subset \overline{\Sigma_{l,\Gr(k,\mathbb C^N)}}$.
	\end{prop}
\begin{proof}
Recall the representatives \eqref{representatives secant} and \eqref{representatives tangent} of the orbits $\Sigma_{l,\Gr(k,\mathbb C^N)}$ and $\Theta_{l,\Gr(k,\mathbb C^N)}$, respectively. It is enough to find elements $g_\epsilon \in \GL(N)$ such that $\lim_{\epsilon \to 0} (g_\epsilon \cdot s_l) = \theta_l$. For any $\epsilon>0$ consider the element $g_\epsilon \in GL(N)$ acting as
\[ e_i\ \longmapsto\ e_i + \varepsilon \cdot e_{k+i},\quad e_{k+j}\ \longmapsto\ e_{k+j-l} +\varepsilon^2 \cdot e_{k+j}, \quad e_{k+l}\ \longmapsto\ -e_{k} +\varepsilon^2 \cdot e_{k+l}\]
for any $i=1:l$ and $j=1: l-1$, and as the identity on the other basis vectors: since the images of the basis vectors are all linearly independent, the linear map $g_\epsilon$ actually belongs to $GL(N)$.
From a straightforward count one gets $g_\epsilon \cdot s_l = \epsilon \cdot \theta_l + \epsilon^2 \cdot \overline{t}$ for a suitable $\overline{t} \in \bigwedge^k\mathbb C^N$, hence $\lim_{\epsilon \to 0} (\epsilon^{-1} g_\epsilon \cdot s_l) = \theta_l$, that is the thesis. 
\end{proof}

We complete Remark \ref{rmk on tangent orbits} by showing that for any $l=3:k$ it holds $\Theta_{l,\Gr(k,\mathbb C^N)} \neq \Sigma_{l,\Gr(k,\mathbb C^N)}$. For the {\em dense} case $l=k$ the equality does not hold since for $k\geq 3$ the secant variety is always non-defective and $\tau (\Gr(k,\mathbb C^N))\subsetneq \sigma_2(\Gr(k,\mathbb C^N))$ \cite[Theorem 1.4]{zak1993tangents}. 

\begin{prop}\label{prop:no collapsing}
	For any $3\leq l \leq k-1$ it holds $\Sigma_{l,\Gr(k,\mathbb C^N)}\neq \Theta_{l,\Gr(k,\mathbb C^N)}$.
	\end{prop}
\begin{proof}
	By contradiction, we assume that there exists $3\leq l \leq k-1$ such that $\Sigma_{l,\Gr(k,\mathbb C^N)}=\Theta_{l,\Gr(k,\mathbb C^N)}$. Then by Propositions \ref{IncOrb:Sec} and \ref{IncOrb:Tang} we easily get that $\Sigma_{i,\Gr(k,\mathbb C^N)} = \Theta_{i,\Gr(k,\mathbb C^N)}$ for any $2 \leq i \leq l$: in particular, $\Sigma_{3,\Gr(k,\mathbb C^N)} = \Theta_{3,\Gr(k,\mathbb C^N)}$. Thus we assume by contradiction that $\Sigma_{3,\Gr(k,\mathbb C^N)} = \Theta_{3,\Gr(k,\mathbb C^N)}$. By hypothesis it holds $k\geq 4$. Consider the tangent representative of $\Theta_{3,\Gr(k,\mathbb C^N)}$
	\begin{align*}
		\theta_3 & = e_2\wedge \ldots \wedge e_k \wedge e_{k+1} + e_1\wedge e_3\wedge  \ldots \wedge e_k \wedge e_{k+2} + e_1\wedge e_2 \wedge e_4 \ldots \wedge e_k \wedge e_{k+3}\\
		& = e_4\wedge \ldots \wedge e_k \wedge \left( \underbrace{e_2\wedge e_3 \wedge e_{k+1} + e_1\wedge e_3 \wedge e_{k+2} + e_1\wedge e_2 \wedge e_{k+3}}_{=:\eta} \right)  \ .
		\end{align*}
	From the multiplication map $\psi_{\theta_3}$ as in \eqref{multiplication map} we get the subspace $H_{\theta_3}=\Ker(\psi_{\theta_3})=\langle e_4,\ldots, e_k\rangle_{\mathbb C}$ having dimension $\dim H_{\theta_3}=k-3$. Define $W:=\langle e_1,e_2,e_3,e_{k+1}, \ldots, e_{N}\rangle_{\mathbb C}$. \\
	\indent Since $\theta_3\in \Theta_{3,\Gr(k,\mathbb C^N)}=\Sigma_{3,\Gr(3,\mathbb C^N)}$, there exist $p,q \in \Gr(k,\mathbb C^N)$ such that $\theta_3=p+q$ and $d(p,q)=3$. Then, by definition as kernels, one gets $\langle e_4,\ldots, e_k\rangle_{\mathbb C}=H_{\theta_3}=H_{p+q}=H_p\cap H_q$, where the last equality follows from Proposition \ref{prop:intersection subspaces to secant pts}. This implies that we can write $p+q=e_4\wedge \ldots \wedge e_k \wedge \big( a + b \big)$ for a certain $a+b \in \Sigma_{3,\Gr(3,W)}$. Given the multiplication map 
	\[\begin{matrix}
		\mu : & \bigwedge^3W & \longrightarrow & \bigwedge^k\mathbb C^N\\
		& t & \mapsto & e_4\wedge \ldots \wedge e_k \wedge t 
		\end{matrix} \ , \]
	it holds $\mu(\eta)=\theta_3=p+q=\mu(a+b)$, and by injectivity of $\mu$ we get $\eta=a+b$. But $\eta$ is exactly the representative of the orbit $\Theta_{3,\Gr(3,W)}$, while $a+b$ is in the orbit $\Sigma_{3,\Gr(3,W)}$: in particular, it follows that $\tau(\Gr(3,W))=\sigma_2(\Gr(3,W))$ which is a contradiction.
	\end{proof}

\begin{teo}\label{thm:orbit stratification}
The secant variety $\sigma_2(\Gr(k,\mathbb C^N))$ has the following $\SL(N)$-orbit partition
\begin{center}
			\begin{tikzpicture}[scale=3]
			
			\node(S) at (0,0.2){{$\Gr(k,\mathbb C^N)$}};
			\node(t2) at (0,0.6){{$\Theta_{2,\Gr(k,\mathbb C^N)} = \Sigma_{2,\Gr(k,\mathbb C^N)}$}};
			\node(t3) at (-0.6,1){{$\Theta_{3,\Gr(k,\mathbb C^N)}$}};
			\node(t) at (-0.6,1.4){{$\vdots$}};
			\node(td) at (-0.6,1.8){$\Theta_{k,\Gr(k,\mathbb C^N)}$};
			
			\node(s3) at (0.6,1.2){{$\Sigma_{3,\Gr(k,\mathbb C^N)}$}};
			\node(s) at (0.6,1.6){{$\vdots$}};
			\node(sd) at (0.6,2){{$\Sigma_{k,\Gr(k,\mathbb C^N)}$}};
			
			\path[font=\scriptsize,>= angle 90]
			(S) edge [->] node [left] {} (t2)
			(t2) edge [->] node [left] {} (t3)
			(t3) edge [->] node [left] {} (s3)
			(t3) edge [->] node [left] {} (t)
			(t) edge [->] node [left] {} (td)
			(t) edge [->] node [left] {} (s)
			(td) edge [->] node [left] {} (sd)
			(s3) edge [->] node [left] {} (s)
			(s) edge [->] node [left] {} (sd);
			\end{tikzpicture}
	\end{center}
where the arrows denote the inclusion of an orbit into the closure of the other one. In particular, $\Theta_{k,\Gr(k,\mathbb C^N)}$ is the dense orbit in $\tau(\Gr(k,\mathbb C^N))$, and $\Sigma_{k,\Gr(k,\mathbb C^N)}$ is the dense orbit in $\sigma_2(\Gr(k,\mathbb C^N))$.
\end{teo}

\paragraph{Orbit dimensions.} Since $k\geq3$, the secant variety is non-defective and the tangential variety has codimension $1$, thus the corresponding dense orbits have dimensions 
\[ \dim \Sigma_{k,\Gr(k,\mathbb C^N)}=\dim \sigma_2(\Gr(k,\mathbb C^N))= 2k(N-k)+1 \ ,\]
\[ \dim \Theta_{k,\Gr(k,\mathbb C^N)}=\dim \tau(\Gr(k,\mathbb C^N))= 2k(N-k) \ .\]

%\indent Recall that a skew-symmetric rank-$r$ tensor is said to be {\em identifiable} if it admits a unique rank-$r$ decomposition, otherwise it is {\em unidentifiable}. The space of all the decompositions of a tensor is called {\em decomposition locus}. A space $X$ is said (un)identifiable if all its elements are (un)identifiable: in particular, an orbit is (un)identifiable if all its elements are so, and by homogeneity, if and only if a representative of it is so. 
%We say that a tangent point is {\em tangential-identifiable} if it lies on a unique tangent line (see Definition \ref{def:cactus}), otherwise it is {\em tangential-unidentifiable} and the locus of tangency is said {\em tangential-locus} (see Definition \ref{def:tangential-locus}): in particular, a tangent orbit is tangential-(un)identifiable if all its elements are so.\\

In the following we determine the dimensions of the secant orbits, while we postpone the computation for the tangent orbits in Section \ref{sec:cactus} as we use results about tangential-identifiability (cf. Proposition \ref{prop:tangent dimensions}). 

\begin{prop}\label{prop:secant dimensions}
	For $l=2:k-1$, the distance-$l$ secant orbit $\Sigma_{l,\Gr(k,\mathbb C^N)}$ has dimension
	\[\dim\Sigma_{l,\Gr(k,\mathbb C^N)}=\begin{cases}
	k(N-k)+l(N-l)-3 & \text{for } l=2 \\
	k(N-k)+l(N-l)+1 & \text{for } l\geq 3 \ . 
	\end{cases}\]
	In particular, for any $l\geq 3$ the closure $\overline{\Theta_{l,\Gr(k,\mathbb C^N)}}$ is a codimension $1$ subvariety of $\overline{\Sigma_{l,\Gr(k,\mathbb C^N)}}$.
\end{prop}
\begin{proof}
	Fix $l=2:k-1$. Consider the fibration
	\[\begin{matrix}
	\xi: & \Sigma_{l,\Gr(k,\mathbb C^N)} & \longrightarrow & \Gr(k-l,\mathbb C^N)\\
	& p+q & \mapsto & H_p\cap H_q
	\end{matrix}\]
	which is well-defined by Proposition \ref{prop:intersection subspaces to secant pts}. Given $\omega, \mathbbm{e}_l$ as in \eqref{representatives secant}, we have $E_l:=H_\omega \cap H_{\mathbbm{e}_l}=\langle e_1,\ldots, e_{k-l}\rangle_{\mathbb C} \in \Gr(k-l,\mathbb C^N)$ and we define $W:=\langle e_{k-l+1},\ldots,e_{N}\rangle_{\mathbb C}\simeq \mathbb C^{N-k+l}$. Then, following the same arguments in the proof of Proposition \ref{prop:no collapsing}, one gets the fibre
	\begin{align*}
	\xi^{-1}(E_l) & = \left\{ p+q \in \Sigma_{l,\Gr(k,\mathbb C^N)} \ | \ H_p\cap H_q=E_l\right\}\\
	& = \left\{ e_1\wedge \ldots \wedge e_{k-l}\wedge (a+b) \ | \ a,b \in \Gr(l,W), \ d(a,b)=l\right\}\\
	& \simeq \Sigma_{l,\Gr(l,W)}
	\end{align*}
	where the last isomorphism is given by restriction of the multiplication map 
	\[e_1\wedge \ldots \wedge e_{k-l}\wedge \bullet : \bigwedge^lW \longrightarrow \bigwedge^k\mathbb C^N \ . \]
	In particular, $\Sigma_{l,\Gr(l,W)}$ is the dense orbit in the secant variety $\sigma_2(\Gr(l,W))$, thus
	\[ \dim \xi^{-1}(E_l) = \begin{cases}
	2 \dim\Gr(2,W) -3 = 4(N-k)-3 & \text{for } l=2\\
	2 \dim \Gr(l,W)+1 = 2l(N-k) +1 & \text{for } l\geq 3 \ .
	\end{cases}\]
	From the fibre dimension theorem we conclude that 
	\[
	\dim \Sigma_{l,\Gr(k,\mathbb C^N)} = \dim \Gr(k-l,\mathbb C^N)+ \dim \xi^{-1}(E_l) = \begin{cases}
	k(N-k) + 2(N-2) - 3 & \text{for } l=2\\
	k(N-k) + l(N-l) +1 & \text{for } l\geq 3 \ .
	\end{cases}
	\]\vskip-0.2cm
\end{proof}

\section{Identifiability in $\sigma_2\big(\Gr(k,\mathbb C^N)\big)$}\label{sec:identifiability of secant orbits}

\indent \indent  Fix $3\leq k\leq \frac{N}{2}$ and $\Gr(k,\mathbb C^N)$. In this section we prove that the orbit $\Sigma_{2,\Gr(k,\mathbb C^N)}$ is unidentifiable while the orbits $\Sigma_{l,\Gr(k,\mathbb C^N)}$ for $l=3:k$ are identifiable. We refer to \cite[Theorem 1.1]{bernardi2018new} for the most updated results about generic identifiability for skew-symmetric tensors when $N\leq 14$.
%\begin{obs*}
%	We say that an orbit $\Sigma_{l,\Gr(k,\mathbb C^N)}$ is (un)identifiable if all its elemements are so. We say that a closure $\overline{\Sigma_{l,\Gr(k,\mathbb C^N)}}$ is {\em generically identifiable} if the elements in $\Sigma_{l,\Gr(k,\mathbb C^N)}$ are so: in particular, the secant variety $\sigma_2(\Gr(k,\mathbb C^N))$ is generically identifiable if its dense orbit $\Sigma_{k,\Gr(k,\mathbb C^N)}$ is identifiable.
%	\end{obs*}

\paragraph{Unidentifiability of $\Sigma_{2,\Gr(k,\mathbb C^N)}$.} The unidentifiability of the distance-$2$ orbit $\Sigma_{2,\Gr(k,\mathbb C^N)}$ is a consequence of the fact that the Grassmannian $\Gr(2,\mathbb C^4)$ is just a quadric in $\mathbb P^5$. It is clear if one considers the representative
\[ \omega + \mathbbm{e}_2= e_1\wedge \ldots \wedge e_{k-2}\wedge \big(e_{k-1}\wedge e_{k} + e_{k+1}\wedge e_{k+2}\big) \ \in \Sigma_{2,\Gr(k,\mathbb C^N)} \ ,\]
and notices that the sum in the round brackets actually is a sum of two points lying on a quadric in $\mathbb P(\bigwedge^2\langle e_{k-1},e_k,e_{k+1},e_{k+2}\rangle_{\mathbb C})$, which is unidentifiable: for instance, another decomposition is
%\indent Let's see this more in detail. The above representative defines the subspace $H_{\omega+\mathbbm{e}_2}=H_\omega \cap H_{\mathbbm{e}_2}=\langle e_1,\ldots, e_{k-2}\rangle_{\mathbb C}$. Denote $W:= \langle e_{k-1},\ldots , e_{N}\rangle_{\mathbb C}$. Consider $p,q \in \Gr(k,\mathbb C^N)$ such that $p+q=\omega+\mathbbm{e}_2$: in particular, $H_p\cap H_q=H_{p+q}=H_{\omega + \mathbbm{e}_2}$, thus we can write $p=e_1\wedge \ldots \wedge e_{k-2}\wedge a$ and $q=e_1\wedge \ldots \wedge e_{k-2}\wedge b$ for suitable $a,b \in \Gr(2,W)$ such that $H_a\cap H_b=\{0\}$. Since the multiplication map $\mu: \bigwedge^2W \rightarrow \bigwedge^k\mathbb C^N$ such that $t \mapsto e_1\wedge \ldots \wedge e_{k-2}\wedge t$ is injective, from $\mu (a+b)=p+q=\omega + \mathbbm{e}_2=\mu(e_{k-1}\wedge e_k+e_{k+1}\wedge e_{k+2})$ implies $a+b=e_{k-1}\wedge e_k+e_{k+1}\wedge e_{k+2}$. In particular, the decomposition locus of $\omega + \mathbbm{e_2}$ in $\Sigma_{2,\Gr(k,\mathbb C^N)}$ is given by the decomposition locus of $e_{k-1}\wedge e_k+e_{k+1}\wedge e_{k+2}$ in $\Sigma_{2,\Gr(2,W)}$. But the latter element is unidentifiable: for instance, another decomposition is given by 
\[ e_{k-1}\wedge e_k+e_{k+1}\wedge e_{k+2}=e_{k-1}\wedge (e_k+e_{k+1}) + e_{k+1}\wedge (e_{k+2} + e_{k-1}) \ .\]
Thus $\omega +\mathbbm{e}_2$ is unidentifiable too, hence the orbit $\Sigma_{2,\Gr(k,\mathbb C^N)}$ is so. Moreover, the dimension of the decomposition locus of $e_{k-1}\wedge e_k+e_{k+1}\wedge e_{k+2}$ in $\Sigma_{2,\Gr(2,W)}\subset \sigma_2^\circ(\Gr(2,W))$ is equal to the defect of $\sigma_{2}(\Gr(2,W))$ which is $4$, since $\dim \sigma_2(\Gr(2,W)))=(2\dim \Gr(2,W)+1)-4$.

\begin{cor}\label{cor:decomposition locus Sigma_2}
	The distance-$2$ orbit $\Sigma_{2,\Gr(k,\mathbb C^N)}$ is unidentifiable. Moreover, the decomposition locus of any point in $\Sigma_{2,\Gr(k,\mathbb C^N)}$ has dimension $4$.
	\end{cor}

\paragraph{Identifiability of $\Sigma_{l,\Gr(k,\mathbb C^N)}$ for $l\geq 3$.} First we prove that the dense $\Sigma_{k,\Gr(k,\mathbb C^N)}$ is identifiable for $k\geq 3$, and then we conclude the identifiability of any orbit $\Sigma_{l,\Gr(k,\mathbb C^N)}$ for $l\geq 3$.

\begin{lemma}\label{lemma:generic identifiability}
	For any $k\geq 3$, the dense orbit $\Sigma_{k,\Gr(k,\mathbb C^N)}$ is identifiable. In particular, the secant variety $\sigma_2(\Gr(k,\mathbb C^N))$ is generically identifiable.
	\end{lemma}	

\begin{proof}
	Consider the multiplication map $\bigwedge^kV \otimes \bigwedge^2V \longrightarrow \bigwedge^{k+2}V$ defined on decomposable elements as $x\otimes y \mapsto x\wedge y$ and extended by linearity. Given a general secant point $p+q \in \Sigma_{k,\Gr(k,\mathbb C^N)}$, the induced multiplication map $(p+q)\wedge \bullet: \bigwedge^2V \longrightarrow \bigwedge^{k+2}V$ has kernel 
	\begin{align*} 
	\ker((p+q)\wedge \bullet) = \left\{ \sum_i v_i\wedge w_i \ | \ \sum_i p\wedge v_i\wedge w_i = - \sum_i q\wedge v_i\wedge w_i\right\} = H_p \wedge H_q 
	\end{align*}
	where the last equality follows from the fact that $d(p,q)=k$, that is $H_p\cap H_q=0$. Then, given the point $p+q$, the subspaces $H_p$ and $H_q$ (hence $p$ and $q$) can be recovered in a unique way as follows. \\
	First, one recovers $H_p\oplus H_q$ as kernel of $V \wedge \ker((p+q)\wedge \bullet) \rightarrow \bigwedge^3V$.\\
	%Then one distinguishes $H_p$ and $H_q$ by multiplying elements of the form $v_1\wedge w_1 + v_2\wedge w_2$ in the kernel with $p+q$:
	%\textcolor{red}{indeed, if  	
	%$0=v\wedge w\wedge (p'+q')$ implies $v\wedge w \wedge p'=v\wedge w \wedge q'=0$, and it must be $v \wedge w \in H_{p'}\wedge H_{q'}$}
	Next, let $p'+q'=p+q$ be another decomposition: since  $H_{p'}\cap H_{q'}\subset H_{p'+q'}=H_{p+q}=H_p\cap H_q=\{0\}$, it holds $d(p',q')=k$. Clearly, as kernel of a multiplication map with respect to $p+q=p'+q'$, it holds $H_p\wedge H_q=H_{p'}\wedge H_{q'}$, as well as $H_p\oplus H_q=H_{p'}\oplus H_{q'}$. In particular, for any $v \in H_p$ it must hold either $v \in H_{p'}$ or $v \in H_{q'}$ (similar for $w\in H_q$): indeed, if $v=v_{p'}+v_{q'}$ and $w=w_{p'}+w_{q'}$ in $H_{p'}\oplus H_{q'}$ for some $v_{p'},v_{q'},w_{p'},w_{q'}\neq 0$, then $0=(v\wedge w)\wedge (p+w)= (v_{p'}\wedge w_{p'}  + v_{p'}\wedge w_{q'} + v_{q'}\wedge w_{p'} + v_{q'}\wedge w_{q'}) \wedge (p'+q')= v_{p'}\wedge w_{p'}\wedge q' + v_{q'}\wedge w_{q'}\wedge p'$, leading to a contradiction since $\dim H_{p'}=\dim H_{q'}=k\geq 3$ (together with $H_{p'}\cap H_{q'}=\{0\}$) implies $\langle H_{p'}, v_{q'},w_{q'}\rangle \neq \langle H_{q'},v_{p'},w_{p'} \rangle$.   \\
	Finally, assume by contradiction that $\{H_p,H_q\}\neq \{H_{p'},H_{q'}\}$, that is there exist $v_1, v_2\in H_p$ such that $v_1\in H_{p'}$ and $v_2\in H_{q'}$, and similarly $w_1,w_2\in H_q$ such that $w_1\in H_{p'}$ and $w_2\in H_{q'}$.Then one gets 
	\begin{align*}
		0 & =(v_1\wedge w_1 + v_2\wedge w_2)\wedge (p+q) =(v_1\wedge w_1 + v_2\wedge w_2)\wedge (p'+q')\\
		& = v_1\wedge w_1\wedge p' + v_2\wedge w_2 \wedge q' \neq 0 \ ,
		\end{align*}
	hence a contradiction. We conclude that $\{H_p,H_q\}=\{H_{p'},H_{q'}\}$, hence $\{p,q\}=\{p',q'\}$.
	\end{proof}

\begin{teo}\label{thm:identifiability for secant orbits}
	For any $k \geq 3$ and any $3 \leq l \leq k$, the secant orbit $\Sigma_{l,\Gr(k,\mathbb C^N)}$ is identifiable.
	\end{teo}
\begin{proof}
	Fix $k\geq 3$ and assume $N \geq 2k$. Since from Lemma \ref{lemma:generic identifiability} we already know that the dense orbit $\Sigma_{k,\Gr(k,\mathbb C^N)}$ is identifiable, we fix $3\leq l \leq k-1$. By homogeneity, it is enough to prove the thesis for the representative 
	\[\omega + \mathbbm{e}_l = e_1\wedge \ldots \wedge e_{k-l} \wedge (e_{k-l+1}\wedge \ldots \wedge e_k + e_{k+1} \wedge \ldots \wedge e_{k+l}) \in \Sigma_{l,\Gr(k,\mathbb C^N)} \ . \]
	Let $p,q \in \Gr(k,\mathbb C^N)$ be such that $p+q=\omega + \mathbbm{e}_l$: in particular, $H_p\cap H_q=H_{p+q}=H_{\omega + \mathbbm{e}_l}=\langle e_1\wedge \ldots \wedge e_{k-l}\rangle_{\mathbb C}$. Given $W:=\langle e_{k-l+1}, \ldots, e_N\rangle_{\mathbb C}\simeq \mathbb C^{N-k+l}$, we can write
	\[ p=e_1\wedge \ldots \wedge e_{k-l}\wedge v_{k-l+1}\wedge \ldots \wedge v_{k}=e_1\wedge \ldots \wedge e_{k-l}\wedge a \]
	\[ q=e_1\wedge \ldots \wedge e_{k-l}\wedge w_{k-l+1}\wedge \ldots \wedge w_{k}=e_1\wedge \ldots \wedge e_{k-l}\wedge b \ , \]
	and it holds $H_a\cap H_b=\{0\}$, that is $a+b\in \Sigma_{l,\Gr(l,W)}$. Now, the multiplication map
	\[ \begin{matrix}
		\mu : & \bigwedge^{l}W & \longrightarrow & \bigwedge^k\mathbb C^N\\
		& t & \mapsto & e_1\wedge \ldots \wedge e_{k-l}\wedge t
		\end{matrix}\]
	restricts to the map $\Sigma_{l,\Gr(l,W)}\rightarrow \Sigma_{l,\Gr(k,\mathbb C^N)}$: since $\mu(a+b)=p+q=\omega+\mathbbm{e}_l=\mu(e_{k-l+1}\wedge \ldots \wedge e_k + e_{k+1} \wedge \ldots \wedge e_{k+l})$, by injectivity we get $a+b= e_{k-l+1}\wedge \ldots \wedge e_k + e_{k+1} \wedge \ldots \wedge e_{k+l} \in \Sigma_{l,\Gr(l,W)}$. But from Lemma \ref{lemma:generic identifiability} $\Sigma_{l,\Gr(l,W)}$ is identifiable, thus $\{a,b\}=\{e_{k-l+1}\wedge \ldots \wedge e_k, e_{k+1} \wedge \ldots \wedge e_{k+l}\}$ and $\{p,q\}=\{\omega, \mathbbm{e}_l\}$, that is $\omega+ \mathbbm{e}_l$ is identifiable.
	\end{proof}

\section{Tangential-identifiability in $\tau(\Gr(k,\mathbb C^N))$}\label{sec:cactus}

\indent \indent In this section we focus on the tangent orbits $\Theta_{l,\Gr(k,\mathbb C^N)}\subset \tau(\Gr(k,\mathbb C^N))$. We point out that each tangent orbit $\Theta_{l,\Gr(k,\mathbb C^N)}$ for $l\geq2$ is {\em unidentifiable}: indeed, any representative $\theta_l\in \Theta_{l,\Gr(k,\mathbb C^N)}$ in \eqref{representatives tangent} admits the equivalent decomposition
	\[ \big(e_{1}\wedge (e_{k+2}+e_{k+1}) + e_{k+1}\wedge (e_{1} + e_{2})\big)\wedge e_3\wedge \ldots \wedge e_k + \sum_{j=3}^l e_1 \wedge \dots \wedge e_{j-1} \wedge e_{k+j} \wedge e_{j+1} \wedge \dots \wedge e_k \ . \]
%\noindent However we can consider a different notion of identifiability for tangent points.
%
%\begin{Def}\label{def:cactus}
%	A tangent point $q \in \tau(\Gr(k,\mathbb C^N))$ is {\em tangential-identifiable} if it lies on a unique tangent line to $\Gr(k,\mathbb C^N)$, or equivalently if there exists a unique $x \in \Gr(k,\mathbb C^N)$ such that $q \in T_x\Gr(k,\mathbb C^N)$. Otherwise it is {\em tangential-unidentifiable}.
%	\end{Def}
%
%\noindent In particular, an orbit $\Theta_{l,\Gr(k,\mathbb C^N)}$ is tangential-(un)identifiable if all its elements are so.
%
%\begin{Def}\label{def:tangential-locus}
%	Given a tangent point $q \in \tau(\Gr(k,\mathbb C^N))$, its {\em tangential-locus} is the set of points $x \in \Gr(k,\mathbb C^N)$ such that $q \in T_{x}\Gr(k,\mathbb C^N)$.
%	\end{Def}
%
%\noindent In particular, if $q$ is tangential-identifiable, then its tangential-locus is given by a single point in $\Gr(k,\mathbb C^N)$.

\begin{obs}\label{rmk:non cactus for distance 2}
	The distance-$2$ orbit $\Sigma_{2,\Gr(k,\mathbb C^N)}=\Theta_{2,\Gr(k,\mathbb C^N)}$ is {\em not} tangential-identifiable: indeed, the representative $\theta_2=e_{k+1}\wedge e_2 \wedge \ldots \wedge e_k + e_1\wedge e_{k+2}\wedge e_3\wedge \ldots \wedge e_k$ lies on both the tangent spaces $T_\omega\Gr(k,\mathbb C^N)$ and $T_{e_3\wedge \ldots \wedge e_{k+2}}\Gr(k,\mathbb C^N)$.\\
%
% Actually, by looking at \eqref{SpTang:Grass}, it is straightforward that the latter tangent spaces are the only ones on which $\theta_2$ lies. \\
	However, for any equivalent decomposition of $\theta_2$ one can exhibit two tangent spaces on which that decomposition lies. Thus we conclude that {\em the tangential-locus of $\theta_2$ has the same dimension of the decomposition locus of $\theta_2$}, which by Corollary \ref{cor:decomposition locus Sigma_2} is $4$.
\end{obs}

\begin{teo}\label{thm:tangential-identifiability for tangent orbits}
	For any $l\geq 3$, the tangent orbit $\Theta_{l,\Gr(k,\mathbb C^N)}$ is tangential-identifiable.
	\end{teo}

\begin{proof}
	Fix $l\geq 3$. By homogeneity, it is enough to prove the thesis for the representative $\theta_l=\sum_{j=1}^l\theta_{l,j} \in \Theta_{l,\Gr(k,\mathbb C^N)}$ where $\theta_{l,j}:=e_1 \wedge \dots \wedge e_{j-1} \wedge e_{k+j} \wedge e_{j+1} \wedge \dots \wedge e_k$ are the summands appearing in \eqref{representatives tangent}. We already know that $\theta_l \in T_\omega\Gr(k,\mathbb C^N)$. We want to prove that, if $q \in \Gr(k,\mathbb C^N)$ is such that $\theta_l \in T_q \Gr(k,\mathbb C^N)$, then $q = \omega$.\\
	\indent Assume $\theta_l \in T_q\Gr(k,\mathbb C^N)$ for some $q \in \Gr(k,\mathbb C^N)$, hence $\theta_l \in T_\omega\Gr(k,\mathbb C^N)\cap T_q\Gr(k,\mathbb C^N)$. Notice that
	\begin{align}
		T_\omega\Gr(k,\mathbb C^N)\cap T_q\Gr(k,\mathbb C^N) & = \left( \bigwedge^{k-1}H_\omega\wedge V \right) \cap \left( \bigwedge^{k-1}H_q\wedge V \right) \nonumber \\
		& = \bigwedge^{k-1}(H_\omega\cap H_q)\wedge V \ + \ \bigwedge^{k-2}(H_\omega \cap H_q)\wedge H_\omega \wedge H_q \ . \label{intersection of tg spaces}
		\end{align}
	If $\dim(H_\omega \cap H_q)\lneq k-2$ (i.e. $d(\omega,q)\geq 3$), then $T_\omega \Gr(k,\mathbb C^N)\cap T_q\Gr(k,\mathbb C^N) =\{0\}$, leading to a contradiction. If $k-2 \leq \dim(H_\omega \cap H_q)\leq k-1$ (i.e. $1\leq d(\omega, q) \leq 2$), as each summand $\theta_{l,j}$ for $j=1:l$ is a simple element in the space \eqref{intersection of tg spaces}, it follows that in each $\theta_{l,j}$ there are at least $k-2$ wedge-entries lying in $H_\omega \cap H_q$, that is $H_\omega \cap H_q\subset \langle e_1,\ldots, \hat{e_j}, e_{k+j}, \ldots e_k\rangle_\mathbb C$. Since $l\geq 3$, one deduces $H_q=H_\omega$, leading to contradiction. We conclude that $\dim(H_\omega \cap H_q)=k$ must hold, that is $q=\omega$.
	\end{proof}

\subsection{Dimensions of tangent orbits}

\indent \indent We are now ready to compute the dimensions of the tangent orbits $\Theta_{l,\Gr(k,\mathbb C^N)}$, completing the description in Section \ref{sec:orbit partition}.

\begin{prop}\label{prop:tangent dimensions}
	For $l=2:k$, the distance-$l$ tangent orbit $\Theta_{l,\Gr(k,\mathbb C^N)}$ has dimension 
	\[\dim\Theta_{l,\Gr(k,\mathbb C^N)}=\begin{cases}
	k(N-k)+2(N-2)-3 & \text{for } l=2 \\
	k(N-k)+l(N-l)   & \text{for } l\geq 3 \ .
	\end{cases} \]
\end{prop}
\begin{proof}
	The tangent bundle on the Grassmannian $\Gr(k,\mathbb C^N)$ is $\cal T_{\Gr(k,\mathbb C^N)}=\cal U^\vee \otimes \cal Q$, where $\cal U$ is the rank-$k$ universal bundle and $\cal Q$ is the rank-$(N-k)$ quotient bundle. In particular, its fibre at $p\in \Gr(k,\mathbb C^N)$ is $(\cal T_{\Gr(k,\mathbb C^N)})_p\simeq \mathbb C^k \otimes \mathbb C^{N-k}$.\\
	Via the identification of the tangent bundle with the tangential variety, the orbit $\Theta_{l,\Gr(k,\mathbb C^N)}$ exactly corresponds to the set of all rank-$l$ matrices in $(\cal T_{\Gr(k,\mathbb C^N)})_p$ as $p \in \Gr(k,\mathbb C^N)$ varies: we recall that the set of rank-$l$ matrices $[\mathbb C^k \otimes \mathbb C^{N-k}]_l$ has dimension $l(N-k+k-l)=l(N-l)$.\\
	\indent From Proposition \ref{thm:tangential-identifiability for tangent orbits}, for $l\geq 3$ the tangential-identifiability implies that 
	\[\dim \Theta_{l,\Gr(k,\mathbb C^N)}=\dim\Gr(k,\mathbb C^N)+\dim [\mathbb C^k\otimes \mathbb C^{N-k}]_l= k(N-k)+l(N-l) \ . \]
	On the other hand, by Remark \ref{rmk:non cactus for distance 2} the orbit $\Theta_{2,\Gr(k,\mathbb C^N)}$ is not tangential-identifiable and the tangential-locus of any element of it is $4$-dimensional. In particular, we get
	\[ \dim\Theta_{2,\Gr(k,\mathbb C^N)}=(\dim\Gr(k,\mathbb C^N)-3)+ \dim [\mathbb C^k\otimes \mathbb C^{N-k}]_2= k(N-k)+2(N-2)-3 \ . \] \vskip-0.4cm
\end{proof}

\indent Notice that the dimension of $\Theta_{2,\Gr(k,\mathbb C^N)}$ obtained above is the same as the dimension of the orbit $\Sigma_{2,\Gr(k,\mathbb C^N)}$ obtained in Proposition \ref{prop:secant dimensions}, in agreement with the equality $\Sigma_{2,\Gr(k,\mathbb C^N)}=\Theta_{2,\Gr(k,\mathbb C^N)}$. Moreover, it is worth remarking that for any $l\geq 3$ the tangent orbit $\Theta_{l,\Gr(k,\mathbb C^N)}$ has codimension $1$ in the secant orbit $\Sigma_{l,\Gr(k,\mathbb C^N)}$.

\section{The $2$-nd Terracini locus of $\Gr(k,\mathbb C^N)$}\label{sec:terracini locus}

\indent \indent In the following we underline a connection between the tangential-locus and the {\em Terracini locus}, the latter being introduced in \cite{ballico2021terracini,ballico2020terracini}. The {\em $r$-th Terracini locus} of a variety $X$ is 
\[ \Terr_r(X):=\overline{\left\{ \{p_1,\ldots, p_r\}\in Hilb_r(X) \ | \ \dim \langle T_{p_1}X,\dots,T_{p_r}X \rangle < \dim \sigma_r(X)\right\}} \ , \]
where $Hilb_r(X)$ is the Hilbert scheme of $0$-dimensional subschemes of $X$ of length $r$ (see Section \ref{sec:singular locus} for details). From \cite[Remark 2.1]{chiantiniciliberto} we also recall that, given a variety $X\subset \mathbb P^m$, its {\em abstract secant variety} is
\[ Ab\sigma_2(X):=\overline{\left\{(x,y,p)\in X^2_{/\mathfrak S_2}\times \mathbb P^m \ | \ p \in \langle x,y \rangle\right\}} \ , \]
and it is endowed with the natural projections $\pi_1:Ab\sigma_2(X)\rightarrow X^2_{/\mathfrak S_2}$ and $\pi_2:Ab\sigma_2(X)\rightarrow \mathbb P^m$ such that $\pi_2\left(Ab\sigma_2(X)\right)=\sigma_2(X)$. Notice that the Hilbert scheme $Hilb_2(X)$ is the blow-up $Bl:Hilb_2(X)\rightarrow X^2_{/\mathfrak S_2}$ of the symmetric square along the diagonal (cf. \cite[Section 3]{vermeire2001some}, \cite[Section 1]{ullery}). 

\begin{teo}\label{thm:terracini locus}
	Let $k \geq 3$ and $N \geq 6$. Then the second Terracini locus $\Terr_2\left(\Gr(k,\mathbb C^N)\right)$ of the Grassmannian $\Gr(k,\mathbb C^N)$ corresponds to the distance-$2$ orbit closure $\overline{\Sigma_{2,\Gr(k,\mathbb C^N)}}$. More precisely, in the above notations, it holds
	\[\Terr_2\left(\Gr(k,\mathbb C^N)\right) = (Bl^{-1}\circ\pi_1\circ\pi_2^{-1})\left(\overline{\Sigma_{2,\Gr(k,\mathbb C^N)}}\right) \ . \]
\end{teo}

\begin{proof}
	Consider a point $p+q\in \sigma_2(\Gr(k,\mathbb C^N))$ for certain $p,q \in \Gr(k,\mathbb C^N)$. We show that $\dim \left\langle T_p \Gr(k,\mathbb C^N), T_q \Gr(k,\mathbb C^N) \right\rangle$ drops only for $p+q \in \overline{\Sigma_{2,\Gr(k,\mathbb C^N)}}=\Gr(k,\mathbb C^N)\sqcup \Sigma_{2,\Gr(k,\mathbb C^N)}$. As in \eqref{intersection of tg spaces}, it holds
	\[ T_p\Gr(k,\mathbb C^N) \cap T_q\Gr(k,\mathbb C^N) = \bigwedge^{k-1}(H_p\cap H_q)\wedge V + \bigwedge^{k-2}(H_p \cap H_q)\wedge H_p \wedge H_q \ . \]
	In particular, for $d(p,q)\geq 3$ one gets $T_p\Gr(k,\mathbb C^N) \cap T_q\Gr(k,\mathbb C^N)=\{0\}$, hence the dimension of the span does not drop. On ther other hand, for $d(p,q)\leq 2$ the above intersection has positive dimension and the dimension of the span drops.
	\end{proof}

\section{The singular locus of $\sigma_2(\Gr(k,\mathbb C^N))$}\label{sec:singular locus}

\indent \indent We are now ready to determine the singular locus of the secant variety $\sigma_2(\Gr(k,\mathbb C^N))$, proving that it exactly coincides with the distance-$2$ orbit closure $\overline{\Sigma_{2,\Gr(k,\mathbb C^N)}}$. As in the previous sections, we assume $3\leq k \leq \frac{N}{2}$.

\begin{obs*}
	The secant variety $\sigma_2(\Gr(3,\mathbb C^6))=\mathbb P^{19}$ is the whole space, hence smooth. As a Legendrian variety, the case $\Gr(3,\mathbb C^6)$ has been widely studied in \cite{LM07} and it has been proven that the tangential variety $\tau(\Gr(3,\mathbb C^6))$ has singular locus coinciding with $\overline{\Sigma_{2,\Gr(3,\mathbb C^6)}}$ (corresponding to $\sigma_+$ in \cite{LM07}).
	\end{obs*}

\indent In the respect of the above remark, from now on we consider $\Gr(k,\mathbb C^N)$ for $k\geq 3$ and $N \geq 7$: for any $N\geq 7$, the secant variety $\sigma_2(\Gr(k,\mathbb C^N))$ does not fill up the ambient space $\mathbb P^{\binom{N}{k}-1}$.\\
In the following we show that the singularity of the distance-$2$ orbit $\Sigma_{2,\Gr(k,\mathbb C^N)}$ in the secant variety follows by its tangential-{\em un}identifiability. First, we recall two general lemmas which are well-known to experts: we assume $X$ to be a irreducible smooth projective variety. The first result is a weaker version of classical Terracini's Lemma (cf. \cite[Theorem 1.3.1]{russo2016}).

\begin{lemma}\label{lemma:tangent space inclusion dec}
	Given any $p, q \in X$ and $p+q \in \sigma_2(X)$, the following inclusion holds 
	\[ \left\langle T_p X, T_q X \right\rangle \subset T_{p+q} \sigma_2(X) \ . \]
\end{lemma}
\begin{proof}
	By symmetry, it is enough to prove the inclusion $T_p X \subset T_{p+q} \sigma_2(X)$. For any curve $\gamma(t) \subset X$ such that $\gamma(0) = p$ and $\gamma'(0) = v \in T_p X$, the curve $\delta(t) := q + \gamma(t) \subset \sigma_2(X)$ is such that $\delta(0)=q + \gamma(0) = p+q$ and $\delta'(0) = \gamma'(0) = v$, that is $v \in T_{p+q} \sigma_2(X)$.
\end{proof}

\begin{lemma}\label{lemma:tangent space inclusion cactus}
	For any $q \in \tau(X)$ and any $p\in X$ such that $q \in T_p X$, it holds 
	\[T_p X \subset T_{q} \sigma_2(X) \ . \]
\end{lemma}
\begin{proof}
	Let $q \in \sigma_2(X)$ and $p\in X$ such that  $q \in T_p X$. Given a tangent vector $v \in T_p X$, the line $\gamma(t) := q + t \cdot v$ lies in the tangent space $T_pX\subset \sigma_2(X)$, where the latter inclusion holds by definition of secant variety. Since $\gamma(0) = q$ and $\gamma'(0) = v$, we conclude that $v \in T_q \sigma_2(X)$.
\end{proof}

\begin{lemma}\label{lemma:distance-2 is singular}
	For any $3 \leq k \leq \frac{N}{2}$, the distance-$2$ orbit is singular in the secant variety, i.e.
	\[ \Sigma_{2,\Gr(k,\mathbb C^N)}\subset \Sing\left( \sigma_2\big(\Gr(k,\mathbb C^N)\big)\right) \ . \]
\end{lemma}
\begin{proof}
	By homogeneity, it is enough to prove the singularity of the representative
	\[ s_2 = \omega + \mathbbm{e}_2 = e_1 \wedge \dots \wedge e_k + e_1 \wedge \dots \wedge e_{k-2} \wedge e_{k+1} \wedge e_{k+2} \ . \]
	From Lemma \ref{lemma:tangent space inclusion dec} we know that $\left\langle T_{\omega} \Gr(k,\mathbb C^N),T_{\mathbbm{e}_2} \Gr(k,\mathbb C^N)\right\rangle \subset T_{s_2} \sigma_2(\Gr(k,\mathbb C^N))$.
	On the other hand, from Remark \ref{rmk:non cactus for distance 2} we deduce that the tangential-locus of $s_2$ contains the points 
	\begin{align*} &p_1 = e_1 \wedge \dots \wedge e_{k-2} \wedge e_{k-1} \wedge e_{k+1}, \quad p_2 = e_1 \wedge \dots \wedge e_{k-2} \wedge e_{k-1} \wedge e_{k+2} \\ 
	&p_3 = e_1 \wedge \dots \wedge e_{k-2} \wedge e_{k} \wedge e_{k+1}, \quad p_4 = e_1 \wedge \dots \wedge e_{k-2} \wedge e_{k} \wedge e_{k+2} \ ,
	\end{align*}
	that is $s_2\in T_{p_i}\Gr(k,\mathbb C^N)$ for any $i=1:4$, hence from Lemma \ref{lemma:tangent space inclusion cactus} we get the inclusions $T_{p_i}\Gr(k,\mathbb C^N) \subset T_{s_2} \sigma_2(\Gr(k,\mathbb C^N))$ for any $i=1:4$. In particular, since $T_{s_2}\sigma_2(\Gr(k,\mathbb C^N))$ is a linear space, it must contain the sum 
	\begin{align} 
	& \ \ \ \ \ \ \ T_{\omega} \Gr(k,\mathbb C^N) + T_{\mathbbm{e}_2} \Gr(k,\mathbb C^N) + T_{p_1} \Gr(k,\mathbb C^N) + T_{p_2} \Gr(k,\mathbb C^N) = \nonumber \\
	& = \left( \bigwedge^{k-1}H_\omega\wedge V \right) + \left( \bigwedge^{k-1}H_{\mathbbm{e}_2}\wedge V \right) + \left( \bigwedge^{k-1}H_{p_1}\wedge V \right) + \left( \bigwedge^{k-1}H_{p_2}\wedge V \right) \ . \label{sum of 4 tg spaces} 
	\end{align}
	Given $E_{k-2}:=\langle e_1,\ldots, e_{k-2}\rangle_{\mathbb C}$, for any $p \in \{\omega, \mathbbm{e_2},p_1,p_2\}$ one has $E_{k-2}\subset H_p$ and
	\[
	\bigwedge^{k-1}H_p\wedge V \ = \  \bigwedge^{k-2}E_{k-2}\wedge \frac{H_p}{E_{k-2}} \wedge \frac{V}{H_p} \ \oplus \ \bigwedge^{k-2}E_{k-2}\wedge \bigwedge^2\frac{H_p}{E_{k-2}} \ \oplus \ \bigwedge^{k-3}E_{k-2}\wedge \bigwedge^2\frac{H_p}{E_{k-2}} \wedge \frac{V}{H_p} \ ,
	\]
	and an easy computation of the generators (and their repetitions among the four tangent spaces) shows that the sum in \eqref{sum of 4 tg spaces} has dimension $(N-k)(4k-4)-4k + 6$. If we prove that 
	\[ (N-k)(4k-4)-4k + 6 \gneq 2(N-k)k + 2 = \dim \sigma_2 (\Gr(k,\mathbb C^N))+1 \]
	we are done. Notice that the above strict inequality is equivalent to $(N-k)(2k-4)-4(k-1) \gneq 0$. Moreover, since $N \geq 2k$, it holds $(N-k)(2k-4)-4(k-1) \geq 2k^2 -8k + 4$. \\
	\indent Now, for $k\geq 4\gneq 2+\sqrt{2}$ one gets $2k^2 -8k + 4 \gneq 0$, while for $k = 3$ and $N \geq 8$ one has $(N-k)(2k-4)-4(k-1) = 2(N-3) - 8 \gneq 0$. Finally, for $k=3$ and $N = 7$, Lemma \ref{lemma:tangent space inclusion cactus} implies that the tangent space $T_{s_2} \sigma_2(\Gr(3,\mathbb C^7))$ must contain the sum
	\[ T_{\omega} \Gr(3,\mathbb C^7) + T_{\mathbbm{e}_2} \Gr(3,\mathbb C^7) + T_{p_1} \Gr(3,\mathbb C^7) + T_{p_2} \Gr(3,\mathbb C^7) + T_{p_3} \Gr(3,\mathbb C^7) + T_{p_4} \Gr(3,\mathbb C^7) \]
	which, by similar computations as above, has dimension $30\gneq 26 =\dim \sigma_2(\Gr(3,\mathbb C^7))+1$. In each one of the above cases we conclude that $s_2$ is singular in $\sigma_2(\Gr(k,\mathbb C^N))$, hence $\Sigma_{2,\Gr(k,\mathbb C^N)}$ is so.
	\end{proof}

\indent We are left with proving the smoothness of all the secant and tangent orbits of distance greater or equal than $3$.

\begin{obs}\label{rmk:smoothness of theta3 is enough}
	In order to prove the inclusion $\Sing(\sigma_2(\Gr(k,\mathbb C^N)))\subset \overline{\Sigma_{2,\Gr(k,\mathbb C^N)}}$, it is enough to prove the smoothness for the distance-3 tangent orbit $\Theta_{3,Gr(k,\mathbb C^N)}$, as it is contained in the closure of all the orbits (both secant and tangent) of distance greater or equal than $3$.
	\end{obs}

\noindent In the following we prove the smoothness in $\sigma_2(\Gr(k,\mathbb C^N))$ of the tangent point
\[ q_3:=e_2\wedge \ldots e_k \wedge e_{k+1}- e_1\wedge e_3 \wedge \ldots \wedge e_k \wedge e_{k+2} + e_1\wedge e_2 \wedge e_4 \wedge \ldots \wedge e_k \wedge e_{k+3} \ \in T_{\omega}\Gr(k,\mathbb C^N) \]
lying in the orbit $\Theta_{3,\Gr(k,\mathbb C^N)}$ (it differs from the representative $\theta_3$ by a sign).\\ \indent Consider the dual space $(\mathbb C^N)^\vee$ with coordinates $(x_1,\ldots, x_N)$ such that $x_i(e_j)=\delta_{ij}$. We denote by $\cal I(q)\subset \bigwedge^\bullet (\mathbb C^N)^\vee$ the ideal of a point $q \in \sigma_2(\Gr(k,\mathbb C^N))$, where $\bigwedge^\bullet (\mathbb C^N)^\vee$ acts on $\bigwedge^k\mathbb C^N$ via the {\em skew-apolarity action} \cite[Definition 4]{ABMM21}: namely, for any $\bold{x}_J:=x_{j_1}\wedge \ldots \wedge x_{j_h}\in \bigwedge^h(\mathbb C^N)^\vee$ and any $\bold{v}_I:=v_{i_1}\wedge \ldots\wedge v_{i_k}\in \bigwedge^k\mathbb C^N$ one has 
\[ \bold{x}_J(\bold{v}_I)=\begin{cases}
0 & \text{if } h\gneq k\\
\sum_{R \subset I,\ |R|=h} (-1)^{sign(R)} \det(x_{j}(v_r)) \bold{v}_{I \setminus R} & \text{if } h\leq k 
\end{cases} \ , \]
\noindent where: the sum runs over all possible ordered subsets $R = \{r_1,\ldots,r_h\}\subset I$; $(x_j(v_{r})) =$ $(x_j(v_{r}))_{j \in J, r\in R}$ denotes the $h \times h$ matrix whose $j$-th row is $(x_j(v_{r_1}) \dots x_j(v_{r_h}))$; $v_{I \setminus R}$ denotes the wedge-product obtained from $v_{i_1} \wedge \dots \wedge v_{i_k}$ by removing the entries whose index appear in $R$; $sign(R)$ is the sign of the permutation that sends the ordered sequence $I$ to the ordered sequence $\{R,I\setminus R\}$ by keeping the order of the elements not contained in $R$.\\
Let $(\cal I(q)^2)_k$ be the $k$-th graded component of its squared ideal and let
\[ (\cal I(q)^2)_k^\perp :=\left\{ v \in \bigwedge^k \mathbb C^N \ | \ \alpha(v)=0 \ \forall \alpha \in \cal I(q)^2 \right\}\]
be the subspace of $\bigwedge^k\mathbb C^N$ orthogonal to $\cal I(q)^2$.

\begin{lemma}\label{lemma:ideal}
	For any $q \in \sigma_2(\Gr(k,\mathbb C^N))$, it holds 
	\[T_q\sigma_2(\Gr(k,\mathbb C^N))\subset (\cal I(q)^2)_k^{\perp} \ . \]
	\end{lemma}
\begin{proof}
	Consider $v \in T_q\sigma_2(\Gr(k,\mathbb C^N))$
	%. By definition, $v$ is a linear combination of directions of curves lying in $\sigma_2(\Gr(k,\mathbb C^N))$ and passing through $q$. As any $\alpha \in (\cal I(q)^2)_k$ is continuous and linear, being a derivation, for our pourposes we may assume $v$ to be 
	being the direction of a curve $\gamma(t)\subset \sigma_2(\Gr(k,\mathbb C^N))$ passing through $\gamma(0)=q$. As any $\alpha \in (\cal I(q)^2)_k$ is continuous and linear, being a derivation, one has
	\begin{align*} 
		\alpha (v) & = \alpha \left( \lim_{t\rightarrow 0}\frac{\gamma(t)-q}{t} \right) \stackrel{cont}{=} \lim_{t\rightarrow 0}\frac{\alpha(\gamma(t)-q)}{t} \stackrel{lin}{=} \frac{d}{dt}\alpha (\gamma(t))_{\big|_{t=0}}= \frac{d\alpha}{d\gamma(t)}_{\big|_{\gamma(0)=q}}\!\!\!\!\!\cdot \underbrace{\frac{d\gamma(t)}{dt}_{\big|_{t=0}}}_{=v} \ . 
		\end{align*}
	Since $\alpha \in \cal (I(q)^2)_k$, we can write $\alpha=\sum_jf_j\wedge g_j$ for some $f_j,g_j \in \cal I(q)$. Hence
	\[ \alpha(v) = \sum_j\left[ \frac{df_j}{d\gamma(t)}\wedge g_j + f_j \wedge \frac{dg_j}{d\gamma(t)} \right]_{\big|_{\gamma(0)=q}} \cdot v = \sum_j \left[ f_j'\wedge \underbrace{g_j(q)}_{=0} - g_j'\wedge \underbrace{f_j(q)}_{=0}\right] \cdot v = 0 \ . \]
	\end{proof}

The next step is to compute the dimension of $(\cal I(q_3)^2)_k^\perp$. As it is clear that the multiplication we consider is the wedge product, in the following we lighten up the notation by omitting the wedges: for instance, $x_{i}x_j$ means $x_i\wedge x_j$. The ideal of $q_3$ is generated by
\begin{align*} 
	\cal I(q_3) = ( & \underbrace{x_{k+4} \ , \ldots, \ x_N}_{(1)} \ , \ \underbrace{x_1x_{k+1} \ , \  x_2x_{k+2} \ , \ x_3,x_{k+3}}_{(2)} \ , \\
	& \underbrace{x_{k+1}x_{k+2} \ , \ x_{k+1}x_{k+3} \ , \ x_{k+2}x_{k+3}}_{(3)} \ ,\\
	& \underbrace{x_2x_{k+1}+x_1x_{k+2} \ , \ x_3x_{k+1}+x_1x_{k+3} \ , \ x_2x_{k+3}+x_3x_{k+2}}_{(4)}  \ ) \subset \bigwedge^\bullet (\mathbb C^N)^\vee \ ,
	\end{align*}
and a direct computation shows that the generators of the squared ideal $\cal I(q_3)^2$ are
\begin{align*}
(1)^2 & \ \ \ \ \  x_ix_j \ \forall i,j\in \{k+4, \ldots, N\} \\
(1)(2) & \ \ \ \ \  x_ix_jx_{k+j} \ \forall i \in \{k+4,\ldots,N\}, \ \forall j \in \{1,2,3\} \\
(1)(3) & \ \ \ \ \ x_ix_{k+j}x_{k+s} \ \forall i \in \{k+4,\ldots,N\}, \ \forall j,s \in \{1,2,3\} \\
(2)^2 & \ \ \ \ \ x_jx_sx_{k+j}x_{k+s} \ \forall j,s \in \{1,2,3\} \\
(2)(3) & \ \ \ \ \ x_jx_{k+1}x_{k+2}x_{k+3} \ \forall j \in \{1,2,3\} \\
(1)(4) & \ \ \ \ \ x_i(x_2x_{k+1}+x_1x_{k+2}) \ , \ x_i(x_3x_{k+1}+x_1x_{k+3}) \ , \ x_i(x_2x_{k+3}+x_3x_{k+2}) \ \forall i \in \{k+4,\ldots, N\}\\
(2)(4) & \ \ \ \ \ x_3x_{k+3}(x_2x_{k+1}+x_1x_{k+2}) \ , \ x_2x_{k+2}(x_3x_{k+1}+x_1x_{k+3}) \ , \ x_1x_{k+1}(x_2x_{k+3}+x_3x_{k+2}) \ .
\end{align*}
Consider the sets of generators $A:=\{(1)^2, (1)(2), (1)(3)\}$ and $B:=\{ A, (2)^2, (2)(3) \}$. In particular, it holds
\[ (\cal I(q_3)^2)_k^\perp \ = \ (B)_k^\perp \cap \big( (1)(4),(2)(4) \big)_k^\perp \ . \]
\noindent First we compute the dimension of $(B)_k^\perp$ and then we compute the linearly independent relations imposed by the generators $(1)(4)$ and $(2)(4)$. Given $[r]:=\{1,\ldots, r\}$ for any $r \in \mathbb N$ and given $s \leq r$, we denote by $\binom{[r]}{s}$ the set of all subsets of $[r]$ having $s$ distinct elements. In the following we distinguish the cases $k=3$ and $k\geq 4$: the arguments are the same although the computations are slightly different as $(\cal I(q_3)^2)_3=\big( A, (1)(4)\big)_3$.

\begin{prop}\label{prop:sing k=3}
	For any $N \geq 7$, the subspace $(\cal I(q_3)^2)_3^\perp$ has dimension $6(N-3)+2$, i.e. 
	\[ \dim (\cal I(q_3)^2)_3^\perp =\dim \sigma_2(\Gr(3,\mathbb C^N)) +1 \ . \]
	\end{prop}
\begin{proof}
	Since $(\cal I(q_3)^2)_3^\perp=\big(A, (1)(4) \big)_3^\perp$, first we study $(A)_3^\perp$ and then we cut it by the relations obtained from $(1)(4)$.\\
	\indent Since $A$ is given by monomials, also $(A)_3^\perp$ has to be spanned by monomials of the form $e_{i_1}\wedge e_{i_2}\wedge e_{i_3}$. Let $\omega=e_1\wedge e_2\wedge e_3$ be the (unique by tangential-identifiability) point of tangency of $q_3$. Set $e_{i}\wedge e_{j}\wedge e_{\ell}$ a possible generator of $(B)_3^\perp$, and $d:=d(\omega,e_{i}\wedge e_{j}\wedge e_{\ell})$.
	\begin{itemize}
		\item If $d=3$, then $\{i,j,\ell\}\subset \{4,\ldots, N\}$. Since $e_i\wedge e_j \wedge e_\ell$ has to vanish on $A$, the conditions from $(1)^2$ impose that at least two indices lie in $\{4,5,6\}$. However, the conditions from $(1)(3)$ imply that there cannot be an index lying in $\{7,\ldots, N\}$. Thus the only possibility is that $\{i,j,\ell\}=\{4,5,6\}$, leading to a unique generator of $(A)_3^\perp$ having distance $3$ from $\omega$.
		\item If $d=2$, then we may assume $i \in \{1,2,3\}$ and $\{j,\ell\}\subset \{4,\ldots, N\}$. The relations from $(1)^2$ imply that $\{j,\ell\}\nsubseteq \{7,\ldots, N\}$. If $j\in \{4,5,6\}$ and $\ell\in \{7,\ldots , N\}$, the conditions $(1)(2)$ impose that $j\neq 3+i$, thus one gets $3\cdot 2\cdot (N-6)$ generators. On the other hand, the case $\{j,\ell\}\subset \{4,5,6\}$ leads to other $3\cdot 3$ generators.
		\item The case $d=1$ leads to $3(N-3)$ generators as from Remark \ref{tangent space:dist} these elements span the tangent space at $\omega$.
		\item The case $d=0$ trivially leads to the generator $\omega$ itself.
		\end{itemize}
	As the above generators are all linearly independent, we get $\dim(A)_3^\perp=9(N-3)-7$.\\
	\indent Next we impose on $(A)_3^\perp$ the equations from $(1)(4)$: since the $3(N-6)$ elements in $(1)(4)$ impose linearly independent relations on $(A)_3^\perp$ we get 
	\[ \dim (\cal I(q_3)^2)_3^\perp  = \dim(B)_3^\perp - \dim\langle (1)(4)\rangle = 9(N-3)-7-3(N-6) = 6(N-3)+2 \ . \]
	\end{proof}

\begin{prop}\label{prop:dim B}
	In the above notation, for $k\geq 4$ the dimension of $(B)_k^\perp$ is
	\[\dim (B)_k^\perp= 5 + 3(N-k-3)(k-1) + 6(k-2) + k(N-k) \ . \]
	\end{prop}
\begin{proof}
	\indent Since $(B)_k\subset \bigwedge^k(\mathbb C^N)^\vee$ is spanned by monomials of the form $x_{i_1}\wedge \ldots \wedge x_{i_k}$,  then $(B)_k^\perp$ has to be spanned by monomials of the form $e_{j_1}\wedge \ldots \wedge e_{j_k}$. More precisely, if $(B)_k=\langle x_{i_1}\wedge \ldots \wedge x_{i_k} \ | \ (i_1 \lneq \ldots \lneq i_k) \in I\rangle$ for a certain subset of ordered $k$-tuples  $I\subset \binom{[N]}{k}$, then $(B)_k^\perp=\langle e_{j_1}\wedge \ldots \wedge e_{j_k} \ | \ (j_1 \lneq \ldots \lneq j_k) \notin I\rangle$. 
	We recall that $q_3 \in T_\omega \Gr(k,\mathbb C^N)\cap  \Theta_{3,\Gr(k,\mathbb C^N)}$ where $\omega=e_1\wedge \ldots \wedge e_k$, and that $\omega$ is the only point of tangency for $q_3$,  by tangential-identifiability. We analyze the possible generators $e_{i_1}\wedge \ldots \wedge e_{i_k}$ of $(B)_k^\perp$ as the Hamming distance with respect to $\omega$ varies. We set $d:=d(\omega, e_{i_1}\wedge \ldots \wedge e_{i_k})$.
	\begin{itemize}
		\item[($(i)$] If $d \geq 4$, we may assume $i_1,\ldots,i_{k-d}\in [k]$ and $i_{k-d+1}, \ldots, i_k \in \{k+1,\ldots, N\}$. The conditions imposed by the generators $(1)^2$ imply that there cannot be two or more indices in $\{i_{k-d+1},\ldots, i_k\}\cap \{k+4,\ldots, N\}$. Moreover, if it was $i_k\in \{k+4\, \ldots, N\}$ and $i_{k-d+1},\ldots,i_{k-1}\in \{k+1,k+2,k+3\}$, then it would be $d=4$ leading to contradiction with the conditions imposed by $(1)(3)$. Finally, it cannot be $i_{k-d+1},\ldots, i_{k}\in \{k+1,k+2,k+3\}$ as it would imply $d=3$. Thus there are no generators in $(B)_k^\perp$ having distance at least $4$ from $\omega$.
		\item[$(ii)$] If $d=3$, we may assume $i_1,\ldots, i_{k-3}\in [k]$ and $i_{k-2},i_{k-1},i_k\in \{k+1,\ldots, N\}$. Again, it has to be $\{i_{k-2},i_{k-1},i_k\}\cap \{k+4,\ldots, N\}=\emptyset$ because of the conditions from $(1)^2$ and $(1)(3)$. On the other hand, for $\{i_{k-2},i_{k-1},i_k\}= \{k+1,k+2, k+3\}$ the conditions from $(2)^2$ and $(2)(3)$ impose $\{i_1,\ldots, i_{k-3}\}=[k]\setminus[3]$. Thus there is only one generator in $(B)_k^\perp$ having distance $3$ from $\omega$, namely $e_{4}\wedge \ldots \wedge e_k\wedge e_{k+1}\wedge e_{k+2}\wedge e_{k+3}$.
		\item[$(iii)$] If $d=2$, we may assume $i_1,\ldots, i_{k-2}\in[k]$ and $i_{k-1},i_k\in \{k+1,\ldots , N\}$. Similarly to the above cases, the conditions from $(1)^2$ impose that $\{i_{k-1},i_k\}\nsubseteq\{k+4,\ldots, N\}$. If $i_k\in \{k+4,\ldots, N\}$ and $i_{k-1}\in \{k+1,k+2,k+3\}$, then the conditions from $(1)(2)$ imply $\{i_1,\ldots, i_{k-2}\}\subset [k]\setminus \{i_{k-1}-k\}$, leading to $(N-k-3)\cdot 3 \cdot (k-1)$ generators. Finally, if $\{i_{k-1},i_k\}\subset\{k+1,k+2,k+3\}$, then $(2)^2$ implies that $\{i_{k-1}-k, i_k-k\}\nsubseteq \{i_1,\ldots, i_{k-2}\}$: in particular, the case $\{i_1,\ldots, i_{k-2}\}=[k]\setminus \{i_{k-1}-k,i_k-k\}$ leads to $3\cdot 1$ generators, while the case $|\{i_{k-1}-k,i_k-k\}\cap \{i_1,\ldots, i_{k-2}\}|=1$ leads to $2\cdot 3\cdot (k-2)$ generators.
		\item[$(iv)$] All of the $k(N-k)$ monomials $e_{i_1}\wedge \ldots \wedge e_{i_k}$ of distance $d=1$ from $\omega$ are generators, as by Remark \ref{tangent space:dist} they span $T_{\omega}\Gr(k,\mathbb C^N)$.
		\item[$(v)$] For $d=0$ there trivially is the generator $\omega$ itself.
	\end{itemize}
	Clearly, all of the above generators are linearly independent, hence the thesis follows.
	\end{proof}
	
\begin{prop}\label{prop:dim of ideal perp}
	For $k\geq 4$, the subspace $(\cal I(q_3)^2)_k^\perp$ has dimension $2k(N-k)+2$, that is
	\[ \dim (\cal I(q_3)^2)_k^\perp =\dim \sigma_2(\Gr(k,\mathbb C^N)) +1 \ . \]
	\end{prop}
\begin{proof}
	In order to compute the dimension of $(\cal I(q_3)^2)_k^\perp=\big((1)(4),(2)(4)\big)_k^\perp \cap (B)_k^\perp$, we cut $(B)_k^\perp$ by the relations in $\big((1)(4),(2)(4)\big)_k$. Thus we determine the generators from $\big((1)(4),(2)(4)\big)_k$ which are linearly independent from $(B)_k^\perp$. Notice that any generator in $(2)(4)$ multiplied by $x_i$ for $i=k+4,\ldots, N$ lies in the ideal generated by $(1)(4)$.\\
	\hfill\break
	\indent Let us start from the relations in $((2)(4))_k$. By symmetry, we may consider the element $x_1x_{k+1}(x_2x_{k+3}+x_3x_{k+2}) \in (2)(4)$. The generators in degree $k$ coming from the above element are obtained by multiplying it with monomials $x_{i_1}\wedge \ldots \wedge x_{i_{k-4}} \in \bigwedge^{k-4}(\mathbb C^N)^\vee$. Clearly, $\{i_1,\ldots, i_{k-4}\}$ cannot contain $1,k+1$, otherwise the multiplication goes to zero. On the other hand, it cannot contain $2,3,k+2,\ldots , N$, otherwise the multiplication would give linear combinations of elements of $B$. Thus the generators in degree $k$ coming from $x_1x_{k+1}(x_2x_{k+3}+x_3x_{k+2})$ are obtained from the $k-3$ monomials with indices $\{i_1,\ldots, i_{k-4}\}\in \binom{[k]\setminus[3]}{k-4}$. By symmetry, the same holds for the other two generators in $(2)(4)$. Moreover, a direct computation shows that the generators of $((2)(4))_k$ obtained as above
	\[ x_1x_{k+1}(x_2x_{k+3}+x_3x_{k+2})\wedge x_{i_1}\cdots x_{i_{k-4}} \ \ \ \forall \{i_1,\ldots, i_{k-4}\}\in \binom{[k]\setminus [3]}{k-4}\]
	\[ x_2x_{k+2}(x_1x_{k+3}+x_3x_{k+1})\wedge x_{j_1}\cdots x_{j_{k-4}} \ \ \ \forall \{j_1,\ldots, j_{k-4}\}\in \binom{[k]\setminus [3]}{k-4}\]
	\[ x_3x_{k+3}(x_2x_{k+1}+x_1x_{k+2})\wedge x_{s_1}\cdots x_{s_{k-4}} \ \ \ \forall \{s_1,\ldots, s_{k-4}\}\in \binom{[k]\setminus [3]}{k-4}\]
	are linearly independent. It follows that $((2)(4))_k$ imposes $3(k-3)$ conditions on $(B)_k^\perp$.\\
	\hfill\break 
	\indent Finally, we focus on the relations from $((1)(4))_k$. By symmetry, we consider the set of elements $x_i(x_1x_{k+2}+x_2x_{k+1})\in (1)(4)$ for $i\in \{k+4,\ldots, N\}$ and we multiply it with a monomial $x_{i_1}\wedge \ldots \wedge x_{i_{k-3}}$. Similarly to the previous argument, the set of indices $\{i_1,\ldots, i_{k-3}\}$ cannot contain $1,2,k+1,\ldots, N$ otherwise we would get linear combinations of elements of $B$. However, in this case one can have $3 \in \{i_1,\ldots, i_{k-3}\}$, thus we get $(N-k-3)(k-2)$ generators of $((1)(4))_k$ from $\{ x_i(x_1x_{k+2}+x_2x_{k+1}) \ | \ i=k+4:N\}$ by multiplying it with the $k-2$ monomials indexed by $\{i_1,\ldots, i_{k-3}\}\subset \{3,\ldots, k\}$. Analogously, the remaining sets of elements $\{ x_j(x_1x_{k+3}+x_3x_{k+1}) \ | \ j=k+4:N\}$ and $\{ x_s(x_3x_{k+2}+x_2x_{k+3}) \ | \ s=k+4:N\}$ in $(1)(4)$ give $(N-k-3)(k-2)$ generators each, after multiplying with monomials indexed by $\{j_1,\ldots, j_{k-3}\}\subset \{2,4,\ldots, k\}$ and $\{s_1,\ldots, s_{k-3}\}\subset \{1,4,\ldots, k\}$ respectively. Again, one has to check possible linear combinations among the above $3(N-k-3)(k-2)$ generators: a direct computation shows that the only possible linear combinations are of the form
	\begin{align*}
		x_i(x_1x_{k+2}+x_2x_{k+1})\wedge x_3x_{i_2}\cdots x_{i_{k-3}} & + x_i(x_1x_{k+3}+x_3x_{k+1})\wedge x_2x_{i_2}\cdots x_{i_{k-3}} +\\
		& + x_i(x_3x_{k+2}+x_2x_{k+3})\wedge x_1x_{i_2}\cdots x_{i_{k-3}} &=0 
		\end{align*}
	as $i\in \{k+4,\ldots, N\}$ and $\{i_2,\ldots, i_{k-3}\}\in \binom{[k]\setminus[3]}{k-4}$ vary. It follows that $((1)(4))_k$ imposes $3(N-k-3)(k-2)-(N-k-3)(k-3)$ conditions on $(B)_k^\perp$.\\
	\hfill\break
	\indent We conclude that the dimension of $(\cal I(q_3)^2)_k^\perp$ is 
	\begin{align*}
		\dim (\cal I(q_3)^2)_k^\perp & = \dim \bigg( (B)_k^\perp \cap \big((1)(4),(2)(4)\big)_k^\perp \bigg)\\
		& =  \big[5 + 3(N-k-3)(k-1) + 6(k-2) + k(N-k)\big] - 3(k-3) \\
		& \ \ \ \ - \big[ 3(N-k-3)(k-2)-(N-k-3)(k-3) \big]\\
		& = 2k(N-k)+2 \ .
		\end{align*}
	\end{proof}

\begin{teo}\label{thm:singular locus}
	For any $3\leq k \leq \frac{N}{2}$ and $N\geq 7$, the singular locus of the secant variety of lines to the Grassmannian $\Gr(k,\mathbb C^N)$ coincides with the closure of the distance-$2$ orbit, i.e.
	\[ \Sing\left(\sigma_2(\Gr(k,\mathbb C^N))\right)=\overline{\Sigma_{2,\Gr(k,\mathbb C^N)}} \ .\]
\end{teo}
\begin{proof}
	From Lemma \ref{lemma:distance-2 is singular} we already know that the inclusion $\overline{\Sigma_{2,\Gr(k,\mathbb C^N)}}\subset \Sing(\sigma_2(\Gr(k,\mathbb C^N)))$ holds. From Remark \ref{rmk:smoothness of theta3 is enough} it is enough to prove the smoothness of $\Theta_{3,\Gr(k,\mathbb C^N)}$ for deducing the smoothness of all the remaining orbits. Moreover, by homogeneity it is enough to check the smoothness of a representative of $\Theta_{3,\Gr(k,\mathbb C^N)}$, say $q_3$. Finally, from Lemma \ref{lemma:ideal}, Proposition \ref{prop:sing k=3} (for $k=3$) and Proposition \ref{prop:dim of ideal perp} (for $k\geq 4$) we get the chain of inequalities
	\[ 2k(N-k)+2 \leq \dim T_{q_3}\sigma_2(\Gr(k,\mathbb C^N)) \leq \dim (\cal I(q_3)^2)_k^\perp = 2k(N-k)+2 \ , \]
	leading to $\dim T_{q_3}\sigma_2(\Gr(k,\mathbb C^N))=\dim \sigma_2(\Gr(k,\mathbb C^N))+1$, hence the point $q_3$ is smooth in the secant variety.
	\end{proof}

\begin{obs*}
	Our Theorem \ref{thm:singular locus} corrects a previous statement in \cite[before Figure 1]{abo2009non} in which the authors states that $\Sing(\sigma_2(\Gr(3,\mathbb C^7))) = \Gr(3,\mathbb C^7)$.
	\end{obs*}

\printbibliography
%\bibliographystyle{alpha}
%\bibliography{Sing_Sec_Grassmannian_REVISED}
%\input{Sing_Sec_Grassmannian_REVISED.bbl}

\end{document}